\documentclass[preprint,review,10pt]{elsarticle}

\usepackage{amsmath,amssymb,amsfonts,amsthm}
\usepackage{hyperref}
\usepackage{graphicx}
\usepackage{mathrsfs}

\usepackage{float}
\usepackage{color}
\usepackage{cases}

\usepackage{pifont}
\usepackage{latexsym}
\usepackage{mathrsfs}
\usepackage{dsfont}
\usepackage{eufrak}
\usepackage{multirow}

\usepackage{graphicx}
\usepackage{hyperref}
\usepackage{color}
\usepackage{diagbox}

\usepackage{enumerate}
\usepackage{cases}
\usepackage{xcolor}
\theoremstyle{thmstyleone}%
\newtheorem{theorem}{Theorem}
\usepackage[justification=centering]{caption}

\theoremstyle{thmstyletwo}%
\newtheorem{remark}{Remark}%

\theoremstyle{thmstylethree}%
\newtheorem{corollary}{Corollary}
\newtheorem{assumption}{Assumption}
\newtheorem{lemma}{Lemma}
\let\mr=\mathrm

\begin{document}
\begin{frontmatter}

\title{Supercloseness of the NIPG method for a singularly perturbed convection diffusion problem on Shishkin mesh\tnoteref{funding} }

\tnotetext[funding]{
This research is supported by National Natural Science Foundation of China (11771257, 11601251), Shandong Provincial Natural Science Foundation, China (ZR2021MA004).
}

\author[label1] {Jin Zhang\corref{cor1}}
\author[label1] {Xiaoqi Ma \fnref{cor2}}
\cortext[cor1] {Corresponding author: jinzhangalex@hotmail.com }
\fntext[cor2] {Email: xiaoqiMa@hotmail.com }
\address[label1]{School of Mathematics and Statistics, Shandong Normal University, Jinan 250014, China}

\begin{abstract}

In this paper, we analyze the supercloseness result of nonsymmetric interior penalty Galerkin (NIPG) method on Shishkin mesh for a singularly perturbed convection diffusion problem. According to the characteristics of the solution and the scheme, a new analysis is proposed. More specifically, Gau{\ss} Lobatto interpolation and Gau{\ss} Radau interpolation are introduced inside and outside the layer, respectively. By selecting special penalty parameters at different mesh points, we further establish  supercloseness of almost $k+1$ order under the energy norm. Here $k$ is the order of piecewise polynomials. 
Then, a simple post-processing operator is constructed. In particular, a new analysis  is proposed for the stability analysis of this operator. On the basis of that, we prove that the corresponding post-processing can make the numerical solution achieve higher accuracy. Finally,  superconvergence can be derived under the discrete energy norm.
These theoretical conclusions can be verified numerically.
\end{abstract}

\end{frontmatter}
%

\section{Introduction}

We consider the following singularly perturbed convection diffusion problem:
\begin{equation}\label{eq:S-1}
\begin{aligned}
& Lu:=-\varepsilon u''(x)+a(x)u'(x)+b(x)u(x)=f(x) \quad \text{$x \in \Omega:= (0,1)$},\\
& u(0)= u(1)=0,
\end{aligned}
\end{equation}
where the perturbation parameter $\varepsilon$ satisfies $0<\varepsilon \ll 1$, and the convection coefficient $a(x)\ge\beta>0$. Without loss of generality, for all $x\in\Omega$, we make an assumption that
\begin{equation}\label{eq:SPP-condition-1}
 b(x)-\frac{1}{2}a'(x)\ge \gamma>0
\end{equation}
with some fixed constant $\gamma$. Here $a(x)$, $b(x)$ and $f(x)$ are sufficiently smooth functions. Because of the presence of the parameter $\varepsilon$,  the solution of \eqref{eq:S-1} usually produces a boundary layer of width $\mathcal{O}(\varepsilon \ln (1/\varepsilon) )$ at $x= 1$. In other words, the exact solution to this problem varies dramatically in the vicinity of $x=1$. Hence, classical numerical methods may no longer be suitable. To deal with this difficulty, researchers have designed an efficient method---layer-adapted mesh, which selects finer meshes inside the layer while general meshes outside the layer. The specific details can be found in \cite{Far1Heg2Mil3:2000-R, Lin1:2003-L, Mil1OPi2Shi3:1997-F, Roo1Sty2Tob3:2008-R,Zha1Ma2Yan3:2021-F}. It is worth noting that in \cite{Lin1Sty2:2001-N}, Lin{\ss} and Stynes pointed out that the standard Galerkin method may be unstable even on the layer adapted mesh. This requires us to consider the corresponding stabilization technology in combination with the layer-adapted mesh.

With the in-depth study of singularly perturbed convection diffusion problems, there appear many stabilization methods, one of which is the nonsymmetric interior penalty Galerkin (NIPG) method. It was first introduced \cite{Wh21:1978-motified} in 1978, that employs penalty parameters to solve the discontinuity of finite element functions across the element interface, and allows the construction of high-order schemes in a natural way to approximate the exact solution. Therefore, it is more flexible than the conforming finite element method. Furthermore, compared with other variants of DG method, such as the symmetric interior penalty Galerkin (SIPG) method and incomplete interior penalty Galerkin (IIPG) method, this method is stable and convergent for any nonnegative penalty parameters. In short, no matter how the nonnegative penalty parameters are selected, the NIPG method satisfies coercivity on any mesh. 
In view of its own advantages, the NIPG method has attracted more and more researchers' attention.
Firstly, Roos and Zarin studied the NIPG method for two-dimensional singularly perturbed problems with different boundary layers on a Shishkin mesh with bilinear elements in \cite{Roo1Zar2-2003:T} and \cite{Zar1Roo2:2005-I} respectively, and proved the uniform error estimate with respect to the singular perturbation parameter under an associated norm of this method. 
Then for a two-dimensional singularly perturbed problem, Roos and Zarin discussed the supercloseness of a combination of the NIPG and the finite element method on Shishkin mesh with bilinear elements \cite{Roo1Zar2:2007-motified}. In their work, the region is divided into coarse and fine parts, and then the finite element method is employed in the fine  part of the mesh, while the NIPG method with bilinear elements is used in the coarse part. Through using the local $L^{2}$-projection operator and the nodal bilinear interpolating function in the coarse and fine parts separately, the supercloseness result of almost $\frac{3}{2}$ order is obtained. 
In addition, for a singularly perturbed convection diffusion problem, a high-order NIPG method based on Shishkin-type meshes was analyzed in \cite{Zhy1Yan2Yin3:2015-H}. It turns out that the method has a uniform convergence of almost $k$ order under the associated norm. In particular, numerical experiments there show that the NIPG method has the supercloseness of $k+1$ order in the energy norm when $k$ is odd. However, the corresponding theoretical analysis is unavailable.
This is because it is difficult to analyze it with conventional methods. Therefore, so far, the theoretical system of supercloseness and superconvergence of NIPG method for singularly perturbed problems has not been completely established.

This paper is the first one, which determines the main analysis difficulties of the NIPG method for a singularly perturbed convection diffusion problem---the convergence analysis of the convection term outside the layer and diffusion term in the layer. On this basis, the idea of applying a special interpolation to discuss the supercloseness of NIPG method on Shishkin mesh is proposed. Briefly speaking, for convergence analysis of convection term, outside the layer, the interpolation is defined as Gau{\ss} Radau interpolation; while for the diffusion term, in the layer, we use Gau{\ss} Lobatto interpolation. Then the uniform supercloseness of almost $k+1$ order is obtained under the NIPG norm. In this process, the specific values of penalty parameters at different mesh points can be derived. Furthermore, combined with this supercloseness, we show how to improve the accuracy of the numerical solution by constructing the post-processing operator $R$. Especially, a new method is introduced for the stability analysis of the operator. Finally, we prove the uniform superconvergence under the discrete energy norm.

The basic framework of this paper is organized as follows: In Section 2, we present \emph{a priori} information of the solution. In addition, the Shishkin mesh suitable for the problem is described, and on the basis of that, the NIPG method for the singularly perturbed problem is introduced. Then by defining a special interpolation, we obtain the supercloseness under the relevant energy norm in Section 3. In Section 4, a post-processing operator is constructed, the related properties of this operator are analyzed, and it is proved that the corresponding post-processing can be used to improve the accuracy of the numerical solution. In Section 5, we prove the uniform superconvergence under the discrete energy norm. At last, theoretical conclusion can be verified numerically in Section 6.

Throughout the paper, $k\ge 1$ is some fixed integer and $C$ denotes a generic positive constant that is independent of $\varepsilon$ and the mesh parameter $N$.
\section{Regularity, Shishkin mesh and the NIPG method}\label{sec:mesh,method}
\subsection{Regularity}
Now, we will present some prior information about the solution, which plays a critical role in the following analysis.
\begin{theorem}
The exact solution $u$ of problem \eqref{eq:S-1} can be decomposed as $u = S + E$, in which the smooth part $S$ and the layer part $E$ satisfy $LS=f$ and $LE=0$, respectively. Then for any positive integer $p$, one has
\begin{equation}\label{eq:decomposition}
\begin{aligned}
&\vert S^{(m)}(x)\vert\le C,\quad
&\vert E^{(m)}(x)\vert\le C\varepsilon^{-m}e^{-\beta(1-x)/\varepsilon},\quad 0\le m\le p,
\end{aligned}
\end{equation}
where $p$ depends on the regularity of the coefficients. In particular, \eqref{eq:decomposition} holds for any $p\in \mathcal{N}$ if $b, c, f\in C^{\infty}(\Omega)$.
\end{theorem}
\begin{proof}
We can draw this conclusion by directly referring to \cite{Roo1Sty2Tob3:2008-R}.
\end{proof}

\subsection{Shishkin mesh}
As mentioned above, the solution of problem \eqref{eq:S-1} exists a boundary layer at $x=1$, where the solution changes rapidly. In order to resolve this layer, we adopt a piecewise uniform mesh---Shishkin mesh. 
Furthermore, the set of mesh points  $\Omega_{N}=\{x_{i}\in \Omega : i=0, 1, \cdots, N\}$ is constructed. Then suppose that 
\begin{equation*}
\mathcal{T}_{N} =\{{I_{i} = [x_{i-1}, x_{i}]: i = 1,\cdots, N}\},
\end{equation*}
is a partition of the domain $\Omega$. Here $h_{i}=x_{i}-x_{i-1}$ represents the length of the interval $I_{i}$ and the general interval is denoted as $I$.

For convenience, the domain $\bar{\Omega}$ is subdivided into two subdomains as $\bar{\Omega}= [0, 1-\tau]\cup [1-\tau, 1]$. Here the transition point $1-\tau$ is defined by
$1-\tau=1-\frac{\sigma \varepsilon}{\beta} \ln N$, which satisfies $\tau \le 1/2$.  Let $N\in \mathbb{N}$, $N\ge 4$, be divisible by $2$ and $\sigma\ge k+1$. Each subdomain is divided into $N/2$ equidistant mesh intervals. Therefore, Shishkin mesh points can be represented as
\begin{equation}\label{eq:Shishkin mesh-Roos}
x_{i}=
\left\{
\begin{aligned}
&\frac{2(1-\tau)}{N}i\quad &&\text{for $i=0,1,...,N/2$},\\
& 1-\tau+\frac{2\tau}{N}(i-\frac{N}{2})&&\text{for $i=N/2+1,...,N$}.
\end{aligned}
\right.
\end{equation}
Clearly $x_{N/2}=1-\tau$. Specially, this mesh becomes a uniform mesh when $1-\tau= 1/2$. 

\begin{assumption}\label{ass:S-1}
We shall assume throughout the paper that
\begin{equation*}
\varepsilon \le C N^{-1},
\end{equation*}
and in practice, it is not a restriction.
\end{assumption}

\subsection{The NIPG method}
Below we provide some basic concepts. For each element $I_{i}, i=1, 2, \cdots, N$, the broken Sobolev space of order $s$, where $s$ is a nonnegative integer, is denoted as 
\begin{equation*}
H^{s}(\Omega, \mathcal{T}_{N}) = \{v\in L^{2}(\Omega): v\vert_{I_{i}}\in H^{s}(I_{i}), \forall i=1, 2,\cdots, N\}.
\end{equation*}
At the same time, we present the corresponding norm and seminorm,
\begin{equation*}
\Vert v\Vert^{2}_{s,\mathcal{T}_{N}}=\sum_{i=1}^{N} \Vert v\Vert^{2}_{s,I_{i}},\quad \vert v\vert^{2}_{s,\mathcal{T}_{N}}= \sum_{i=1}^{N}\vert v\vert^{2}_{s,I_ {i}},
\end{equation*}
where $\Vert\cdot\Vert_{s, I_i}$ and $\vert\cdot\vert_{s, I_i}$ are the usual Sobolev norm and semi-norm in $H^s(I_i)$, respectively. 
It is worth noting that $\Vert\cdot\Vert_{I}$ and $(\cdot, \cdot)_{I}$ are usually used to stand for the $L^{2}(I)$-norm and the $L^{2}(I)$-inner product, respectively. Then on Shishkin mesh, the finite element space can be written as
\begin{equation*}
V_{N}^{k}=\{w\in L^{2}(\Omega): w\vert_{I_{i}}\in \mathcal{P}_{k}(I_{i}),\quad \forall i=1, 2, \cdots, N\},
\end{equation*}
where $\mathcal{P}_{k}(I_{i})$ represents the space of polynomials of degree at most $k$ on $I_{i}$. The functions in $V_{N}^{k}$ are completely discontinuous at the boundaries between two adjacent cells.

Since discontinuities are allowed, the function $w\in V_{N}^{k}$ is multivalued at the node. For a function $u\in H^{1}(\Omega, \mathcal{T}_{N})$, the jump and average at the interior node $x_{i}$ are defined by
\begin{equation*}
[u(x_{i})]=u(x_{i}^{-})-u(x_{i}^{+}),\quad \{u(x_{i})\}=\frac{1}{2}\left(u(x_{i}^{+})+u(x_{i}^{-})\right),\quad \forall i=1, \cdots, N-1,
\end{equation*}
with $u(x_{i}^{+})= \lim\limits_{x\rightarrow x_{i}^{+}}u(x)$ and $u(x_{i}^{-})=\lim\limits_{x\rightarrow x_{i}^{-}}u(x)$.
By convention, we extend the definition of jump and average at the boundary nodes $x_{0}$ and $x_{N}$ as
\begin{equation*}
[u(x_{0})]=-u(x_{0}^{+}),\quad \{u(x_{0})\}=u(x_{0}^{+}),\quad [u(x_{N})]=u(x_{N}^{-}),\quad \{u(x_{N})\}=u(x_{N}^{-}).
\end{equation*}

Now the weak formulation of \eqref{eq:S-1} reads as: Find $u_{N} \in V_{N}^{k}$ to satisfy
\begin{equation}\label{eq:SD}
B(u_{N},v_{N})=L(v_{N}) \quad \forall v_{N} \in V_{N}^{k},
\end{equation}
where
\begin{equation*}
\begin{aligned}
&B(u,v)=B_{1}(u,v)+B_{2}(u,v)+B_{3}(u,v),\\
&B_{1}(u,v)=\sum_{i=1}^{N}\int_{I_{i}}\varepsilon u'v'\mr{d}x-\varepsilon \sum_{i=0}^{N}\{u'(x_{i})\}[v(x_{i})]+\varepsilon \sum_{i=0}^{N} \{v'(x_{i})\}[u(x_{i})]+\sum_{i=0}^{N} \rho(x_{i})[u(x_{i})][v(x_{i})],\\
&B_{2}(u,v)=\sum_{i=1}^{N}\int_{I_{i}}a(x) u'v\mr{d}x-\sum_{i=0}^{N-1}a(x_{i})[u(x_{i})]v(x_{i}^{+}),\\
&B_{3}(u,v)=\sum_{i=1}^{N}\int_{I_{i}}b(x)uv\mr{d}x,\\
&L(v)=\sum_{i=1}^{N}\int_{I_{i}}f v\mr{d}x.
\end{aligned}
\end{equation*}
It is worth noting that the penalty  parameters $\rho(x_{i}) (i=0, \cdots, N)$ are nonnegative constants associated with the node $x_{i}$. In the following, we will provide the exact values of $\rho(x_{i})$ as
\begin{equation}\label{eq: penalization parameters}
\rho(x_{i})=\left\{
\begin{aligned}
&1,\quad 0\le i\le N/2-1,\\
&N^{2},\quad N/2\le i\le N.
\end{aligned}
\right.
\end{equation}

\begin{lemma}\label{Galerkin orthogonality property}
Suppose that $u$ is the solution of problem \eqref{eq:S-1}, then the bilinear form $B(\cdot,\cdot)$ defined in \eqref{eq:SD} has the Galerkin orthogonality property
\begin{equation*}
B(u-u_{N}, v)=0,\quad \forall v\in V_{N}^{k}.
\end{equation*}
\end{lemma}
\begin{proof}
From \cite{Zhy1Yan2Yin3:2015-H}, the proof of this lemma can be obtained directly.
\end{proof}

Then we define the energy norm associated with $B(\cdot,\cdot)$ by
\begin{equation}\label{eq:SS-1}
\Vert v \Vert_{NIPG}^{2}:=
\varepsilon \sum_{i=1}^{N}\Vert v'\Vert_{I_{i}}^{2}+\sum_{i=1}^{N}\gamma\Vert v \Vert_{I_{i}}^{2}+\sum_{i=0}^{N}\rho(x_{i})[v(x_{i})]^{2} \quad \forall v\in V_{N}^{k}.
\end{equation}
From \eqref{eq:SPP-condition-1}, some direct calculations show that
\begin{equation}\label{eq:coercity}
B(v_N,v_N) \ge \Vert v_N \Vert_{NIPG}^2\quad \text{for all $v_N\in V_{N}^{k}$},
\end{equation}
which follows that $u_{N}$ is well defined by \eqref{eq:SD}.
%
%
%

\section{Interpolation and supercloseness}

\subsection{Interpolation}
In this section, a special interpolation $\Pi u$ of the exact solution $u$ is introduced, to be brief, 
\begin{equation}\label{eq:H-1}
(\Pi u)\vert_{I}=\left\{
\begin{aligned}
& (I_{k}u)\vert_I, \quad \text{if $I\subset [x_{N/2}, 1]$},\\
& (P^{-}_{h}u)\vert_I,\quad \text{if $I\subset [0, x_{N/2}]$},
\end{aligned}
\right.
\end{equation}
where $P^{-}_{h}u$ is the Gau{\ss} Radau interpolation of $u$ in the finite element space, while $I_{k}u$ represents the   Lagrange interpolation of order $k$ to $u$ at Gau{\ss} Lobatto points. In the subsequent, we will present the definitions and corresponding properties of both interpolations. 

On the one hand, on $I_{i}=[x_{i-1}, x_{i}]$, we suppose that $x_{i-1}= t_{0}< t_{1}< \cdots < t_{k}=x_{i}$ is the Gau{\ss} Lobatto points, where $t_{1}, t_{2}, \cdots, t_{k-1}$ are zeros of the derivative of the Legendre polynomial of degree $k$ on $I_{i}$. For $v\in C(\overline{\Omega})$, $(I_{k} v)\vert_{I_i}$ for $i=1,\ldots,N$ is the $k$-degree Lagrange interpolation polynomial with the Gau{\ss} Lobatto points $\{t_{s}\}_{s=0}^{k}$ as the interpolation points.
Then from the Bramble-Hilbert Lemma \citep[Chapter 4]{cia1:2002-modified} and a scaling argument, we derive 
\begin{equation}\label{eq: Gauss-Lobatto-1}
\vert(z'-I_{k} z', v')_{I_{i}}\vert\le Ch_{i}^{k+1}\vert z\vert_{k+2, I_{i}}\vert v\vert_{1, I_{i}}\quad\forall v\in \mathcal{P}_{k},
\end{equation}
with $z(x)\in H^{k+2}(I_{i})$; see \cite{Zha1:2002-F} for details.

On the other hand, when $k\ge 1$,  Gau{\ss} Radau interpolation $P^{-}_{h}u\in V_{h}^{k}$ is  defined by (see \cite{Che1Shu2:2010-S}): For $i=1, 2, \cdots, N$, 
\begin{align}
&\int_{I_i}(P^{-}_{h}u)v_{h}\mr{d}x=\int_{I_i}uv_{h}\mr{d}x,\quad \forall v_{h}\in \mathcal{P}_{k-1},\label{eq:J-1}\\
&(P^{-}_{h}u)(x_{i}^{-})=u(x_{i}^{-}).\label{eq:J-2}
\end{align}

\begin{remark}\label{special}
Let's analyze the case at the point $x_{N/2}$. For one thing, at the left side of $x_{N/2}$, we use Gau{\ss} Radau interpolation. From \eqref{eq:J-2},
\begin{equation}\label{eq: Kk-11}
(u-P_{h}^{-} u)(x_{N/2}^{-})=0
\end{equation}
can be obtained. For another, at the right side of $x_{N/2}$, from the property of Gau{\ss} Lobatto interpolation, one has
\begin{equation}\label{eq: Kk-12}
(u-I_{k} u)(x_{N/2}^{+})=0.
\end{equation}
Then, combined with \eqref{eq:H-1}, \eqref{eq: Kk-11} and \eqref{eq: Kk-12}, at the point of $x_{N/2}$, there is
\begin{equation*}
[(u-\Pi u)(x_{N/2})]=0.
\end{equation*}
\end{remark}

\subsection{Supercloseness}
Recall that $I_{k}$ is the Lagrange interpolation operator with Gau{\ss} Lobatto points as the interpolation nodes. 
According to the interpolation theories in Sobolev spaces \citep[Theorem 3.1.4]{cia1:2002-modified}, there is
\begin{equation}\label{eq:interpolation-theory}
\Vert \omega-I_{k}\omega \Vert_{W^{l,q}(I_i)}\le C h_i^{k+1-l+1/q-1/p}\vert \omega \vert_{W^{k+1,p}(I_i)}, 
\end{equation}
for all $\omega\in W^{k+1,p}(I_i)$,  where $l=0,1$ and $1\le p,q\le \infty$.   
Recall that $P_{h}^{-}\omega$ is Gau{\ss} Radau interpolation to $\omega$. Then from the projection results \cite{cia1:2002-modified}, we have
\begin{equation}\label{eq:interpolation-theory-1}
\Vert \omega-P_{h}^{-}\omega \Vert_{ I_i}+h_{i}^{\frac{1}{2}}\Vert \omega-P_{h}^{-}\omega \Vert_{L^{\infty}(I_i)}\le C h_i^{k+1}\vert \omega \vert_{k+1, I_i}, \quad i=1, 2, \cdots, N
\end{equation}
for all $\omega\in H^{k+1}(I_{i})$. And on that basis, we have the following lemma.

\begin{lemma}
Let Assumption \ref{ass:S-1} hold and $\sigma\ge k+1$. Recall $\rho(x_{i}), i=0, 1, \cdots, N$ have been presented in \eqref{eq: penalization parameters}. Then on Shishkin mesh \eqref{eq:Shishkin mesh-Roos}, one has
\begin{align}
&\Vert S-P_{h}^{-}S\Vert_{[0, x_{N/2}]}+\Vert u-\Pi u\Vert_{[0, x_{N/2}]}\le CN^{-(k+1)},\label{eq:QQ-1}\\
&\Vert E-I_{k}E\Vert_{[x_{N/2}, 1]}+\Vert u-\Pi u\Vert_{[x_{N/2}, 1]}\le C\varepsilon^{\frac{1}{2}}(N^{-1}\ln N)^{k+1},\label{eq:QQ-2}\\
&\Vert (u-\Pi u)'\Vert_{[0, x_{N/2}]}\le CN^{-k}+C\varepsilon^{-\frac{1}{2}}N^{-\sigma}+CN^{1-\sigma},\label{eq:QQ-3}\\
&\Vert u-\Pi u\Vert_{L^{\infty}(I_{i})}\le CN^{-(k+1)},\quad i=1, 2, \cdots, N/2\label{eq:QQ-5}\\
&\Vert u-\Pi u\Vert_{L^{\infty}(I_{i})}\le C(N^{-1}\ln N)^{k+1},\quad i=N/2+1, \cdots, N\label{eq:QQ-6},\\
&\Vert I_{k}u-u\Vert_{NIPG, [0, x_{N/2}]}\le C\varepsilon^{\frac{1}{2}}N^{-k}+CN^{-(k+1)}+C\varepsilon^{\frac{1}{2}}N^{1-\sigma},\label{post-process-1}\\
&\Vert u-P_{h}^{-}u\Vert_{NIPG, [0, x_{N/2}]}\le CN^{-(k+\frac{1}{2})}.\label{post-process-2}
\end{align}
\end{lemma}
\begin{proof}
Using \eqref{eq:interpolation-theory} and \eqref{eq:interpolation-theory-1}, we can obtain (\ref{eq:QQ-1}-\ref{eq:QQ-6}) simply. In the following, we just analyze \eqref{post-process-1} and \eqref{post-process-2}.

From the definition of the NIPG norm \eqref{eq:SS-1}, properties of Gau{\ss} Lobatto interpolation and Remark \ref{special}, $[(I_{k}u-u)(x_{i})]=0, i=0, 1, \cdots, N/2$. Further, we have
\begin{equation*}
\begin{aligned}
\Vert I_{k}u-u\Vert_{NIPG, [0,x_{N/2}]}^{2}&=\varepsilon \sum_{i=1}^{N/2}\Vert (I_{k}u-u)'\Vert_{I_{i}}^{2}+\sum_{i=1}^{N/2}\gamma\Vert I_{k}u-u\Vert_{I_{i}}^{2},
\end{aligned}
\end{equation*}
where the triangle inequality, the inverse inequality and \eqref{eq:interpolation-theory} yield
\begin{equation*}
\begin{aligned}
\vert\varepsilon \sum_{i=1}^{N/2}\Vert (I_{k}u-u)'\Vert_{I_{i}}^{2}\vert
&\le C\varepsilon \sum_{i=1}^{N/2}\Vert (I_{k}S-S)'\Vert_{I_{i}}^{2}+C\varepsilon \sum_{i=1}^{N/2}\Vert (I_{k}E-E)'\Vert_{I_{i}}^{2}\\
&\le C\varepsilon \sum_{i=1}^{N/2}h_{i}^{2k}\Vert S^{(k+1)}\Vert_{I_{i}}^{2}+C\varepsilon \sum_{i=1}^{N/2}\left(\Vert (I_{k}E)'\Vert_{I_{i}}^{2}+\Vert E'\Vert_{I_{i}}^{2}\right)\\
&\le C\varepsilon\sum_{i=1}^{N/2}h_{i}^{2k+1}\Vert S^{(k+1)}\Vert_{L^{\infty}(I_{i})}^{2}+C\varepsilon\sum_{i=1}^{N/2}N\Vert E\Vert_{L^{\infty}(I_{i})}^{2}+CN^{-2\sigma}\\
&\le C\varepsilon N^{-2k}+C\varepsilon N^{2-2\sigma}+CN^{-2\sigma}.
\end{aligned}
\end{equation*}
In addition, by using \eqref{eq:interpolation-theory} and the triangle inequality, one has
\begin{equation*}
\vert\sum_{i=1}^{N/2}\gamma \Vert I_{k}u-u\Vert_{I_{i}}^{2}\vert\le C\sum_{i=1}^{N/2}\Vert I_{k}S-S\Vert_{I_{i}}^{2}+C\sum_{i=1}^{N/2}\Vert I_{k}E-E\Vert_{I_{i}}^{2}\le CN^{-(k+1)}.
\end{equation*}
Thus we complete the derivation of \eqref{post-process-1}.

Then, according to \eqref{eq:SS-1} and Remark \ref{special}, we have
\begin{equation*}
\begin{aligned}
\Vert u-P_{h}^{-}u\Vert_{NIPG, [0,x_{N/2}]}^{2}&=\varepsilon \sum_{i=1}^{N/2}\Vert (u-P_{h}^{-}u)'\Vert_{I_{i}}^{2}+\sum_{i=1}^{N/2}\gamma\Vert u-P_{h}^{-}u\Vert_{I_{i}}^{2}\\&+\sum_{i=0}^{N/2-1}\rho(x_{i})[(u-P_{h}^{-}u)(x_{i})]^{2}\\&=P_{1}+P_{2}+P_{3}.
\end{aligned}
\end{equation*}
Therefore, we need to analyze $P_{1}$, $P_{2}$, $P_{3}$ in turn.

For $P_{1}$, through \eqref{eq:interpolation-theory-1} and \eqref{eq:decomposition}, there is
\begin{equation*}
\vert\varepsilon\sum_{i=1}^{N/2}\Vert(S-P_{h}^{-}S)'\Vert_{I_{i}}^{2}\vert\le C\varepsilon \sum_{i=1}^{N/2}h_{i}^{2k}\Vert S^{(k+1)}\Vert^{2}_{I_{i}}\le C\varepsilon N^{-2k},
\end{equation*}
while from triangle inequality and inverse inequality,
\begin{equation*}
\begin{aligned}
\vert\varepsilon\sum_{i=1}^{N/2}\Vert(E-P_{h}^{-}E)'\Vert_{I_{i}}^{2}\vert&\le
C\varepsilon \sum_{i=1}^{N/2}\left(\Vert E'\Vert_{I_{i}}^{2}+\Vert (P_{h}^{-}E)'\Vert_{I_{i}}^{2}\right)\\
&\le C\varepsilon \sum_{i=1}^{N/2}\int_{I_{i}}\varepsilon^{-2}e^{-2\beta(1-x)/\varepsilon}\mr{d}x+C\varepsilon\sum_{i=1}^{N/2} h_{i}^{-1}\Vert P_{h}^{-}E\Vert_{L^{\infty}(I_{i})}^{2}\\
&\le C\varepsilon^{-1}\int_{0}^{x_{N/2}}e^{-2\beta(1-x)/\varepsilon}\mr{d}x+C\varepsilon N^{2} N^{-2\sigma}\\
&\le CN^{-2\sigma}+C\varepsilon N^{2}N^{-2\sigma}.
\end{aligned}
\end{equation*}

For $P_{2}$, the triangle inequality and \eqref{eq:interpolation-theory-1} yield
\begin{equation*}
\begin{aligned}
&\vert\sum_{i=1}^{N/2} \gamma \Vert S-P_{h}^{-}S\Vert_{I_{i}}^{2}\vert\le C\sum_{i=1}^{N/2} h_{i}^{2(k+1)}\Vert S^{(k+1)}\Vert_{I_{i}}^{2}\le CN^{-2(k+1)},\\
&\vert\sum_{i=1}^{N/2}\gamma \Vert E-P_{h}^{-}E \Vert_{I_{i}}^{2}\vert\le C\sum_{i=1}^{N/2}\left(\Vert E\Vert_{I_{i}}^{2}+\Vert P_{h}^{-}E\Vert_{I_{i}}^{2}\right)\le CN^{-2\sigma}.
\end{aligned}
\end{equation*}

Finally, using the similar method, one derives
\begin{equation*}
\begin{aligned}
&\vert\sum_{i=1}^{N/2-1}\rho(x_{i})[(S-P_{h}^{-}S)(x_{i})]^{2}\vert\le C\sum_{i=1}^{N/2-1}\rho(x_{i})\Vert S-P_{h}^{-}S\Vert_{L^{\infty}(I_{i}\cup I_{i+1})}^{2}\le CN^{-(2k+1)},\\
&\vert\sum_{i=1}^{N/2-1}\rho(x_{i})[(E-P_{h}^{-}E)(x_{i})]^{2}\vert\le C\sum_{i=1}^{N/2-1}\rho(x_{i})\Vert E-P_{h}^{-}E\Vert_{L^{\infty}(I_{i}\cup I_{i+1})}^{2}\le CN N^{-2\sigma}.
\end{aligned}
\end{equation*}
directly. At this time, \eqref{post-process-2} can be obtained through some simple calculations.
\end{proof}

In the following, we shall prove some error estimates for the interpolation on the element boundaries.  For this purpose, the following multiplicative trace inequality is introduced. 
\begin{lemma}\label{trace inequality}
Assume that $w\in H^{1}(I_{i}),\quad i=1, 2, \cdots, N$, then 
\begin{equation*}
\vert w(x_{s})\vert^{2}\le 2\left(h_{i}^{-1}\Vert w\Vert_{I_{i}}^{2}+\Vert w\Vert_{I_{i}}\Vert w'\Vert_{I_{i}}\right),\quad s\in\{i-1, i\}.
\end{equation*}
\end{lemma}
\begin{proof}
The relevant proof has been presented in \citep[Lemma 4]{Zhu1Xie2Zho3:2011-motified}.
\end{proof}

For the convenience of analysis, $(\Pi u-u)(x)$ is simplified as $\eta(x)$.
\begin{lemma}
Let Assumption \ref{ass:S-1} hold true and $\sigma\ge k+1$. Then on Shishkin mesh \eqref{eq:Shishkin mesh-Roos}, one has
\begin{equation}\label{eq:KK-1}
\{\eta'(x_{i})\}^{2}\le
\left\{
\begin{aligned}
&CN^{-2k}+CN^{-(k+\frac{1}{2})}\varepsilon^{-\frac{3}{2}}N^{-\sigma}+C\varepsilon^{-2}N^{-2\sigma},\quad 0\le i\le N/2-1,\\
&C\varepsilon^{-2}(N^{-1}\ln N)^{2k},\quad N/2\le i\le N.
\end{aligned}
\right.
\end{equation}
\end{lemma}
\begin{proof}
This conclusion can be obtained by means of \cite{Zhy1Yan2Yin3:2015-H}.
From the definition of average and Lemma \ref{trace inequality},
\begin{equation*}
\begin{aligned}
\{\eta'(x_{i})\}^{2}&=\frac{1}{4}\left(\eta'(x^{+}_{i})+\eta'(x^{-}_{i})\right)^{2}\le \frac{1}{2}\left(\eta'(x^{+}_{i})^{2}+\eta'(x^{-}_{i})^{2}\right)\\
&\le h_{i}^{-1}\Vert \eta' \Vert_{I_{i}}^{2}+\Vert \eta' \Vert_{I_{i}}\Vert \eta'' \Vert_{I_{i}}+h_{i+1}^{-1}\Vert \eta' \Vert_{I_{i+1}}^{2}+\Vert \eta' \Vert_{I_{i+1}}\Vert \eta'' \Vert_{I_{i+1}}.
\end{aligned}
\end{equation*}
Now we will estimate $\Vert \eta' \Vert_{I_{i}}$ and $\Vert \eta''\Vert_{I_{i}}$ respectively.

For $i=0, 1, \cdots, N/2-1, N/2$, recalling $h_{i}^{-1}=N$, then we have
\begin{equation}\label{eq: FF-1}
\begin{aligned}
\Vert \eta'\Vert_{I_{i}}^{2}&\le \Vert (S-P_{h}^{-} S)'\Vert^{2}_{I_{i}}+\Vert (E-P_{h}^{-} E)'\Vert_{I_{i}}^{2}\\
&\le Ch_{i}^{2k}\Vert S^{(k+1)}\Vert^{2}_{I_{i}}+\Vert E'\Vert^{2}_{I_{i}}+\Vert (P_{h}^{-} E)'\Vert^{2}_{I_{i}}\\
&\le Ch_{i}^{2k+1}\Vert S^{(k+1)}\Vert_{L^{\infty}(I_{i})}^{2}+C\varepsilon^{-2}\int_{I_{i}}e^{-2\beta(1-x)/\varepsilon}\mr{d}x+Ch_{i}^{-2}\Vert P_{h}^{-} E\Vert_{I_{i}}^{2}\\
&\le CN^{-(2k+1)}+C\varepsilon^{-1}N^{-2\sigma}+Ch_{i}^{-1}\Vert E\Vert^{2}_{L^{\infty}(I_{i})}\\
&\le CN^{-(2k+1)}+C\varepsilon^{-1}N^{-2\sigma}+CN N^{-2\sigma}\\
&\le CN^{-(2k+1)}+C\varepsilon^{-1}N^{-2\sigma},
\end{aligned}
\end{equation}
where the general interpolation theory \eqref{eq:interpolation-theory-1} and the inverse inequality \citep[Theorem 3.2.6]{cia1:2002-modified} have been used. Similarly, we have 
\begin{equation}\label{eq: FF-2}
\Vert \eta''\Vert_{I_{i}}^{2}\le \Vert (S-P_{h}^{-} S)''\Vert^{2}_{I_{i}}+\Vert (E-P_{h}^{-} E)''\Vert_{I_{i}}^{2}\le CN^{-(2k-1)}+C\varepsilon^{-3}N^{-2\sigma}.
\end{equation}
Therefore, for $i=0, 1, \cdots, N/2-1$, the following estimate can be derived though some simple calculations.
\begin{equation*}
\begin{aligned}
\{\eta'(x_{i})\}^{2}&\le h_{i}^{-1}\Vert \eta' \Vert_{I_{i}}^{2}+\Vert \eta' \Vert_{I_{i}}\Vert \eta'' \Vert_{I_{i}}+h_{i+1}^{-1}\Vert \eta' \Vert_{I_{i+1}}^{2}+\Vert \eta' \Vert_{I_{i+1}}\Vert \eta'' \Vert_{I_{i+1}}\\
&\le CN^{-2k}+CN^{-(k+\frac{1}{2})}\varepsilon^{-\frac{3}{2}}N^{-\sigma}+C\varepsilon^{-2}N^{-2\sigma}.
\end{aligned}
\end{equation*}

Furthermore, when $i=N/2+1, \cdots, N$, on the one hand, through \eqref{eq:interpolation-theory} and \eqref{eq:decomposition},
\begin{equation*}
\Vert \eta'\Vert_{I_{i}}^{2}\le \Vert (S-I_{k} S)'\Vert^{2}_{I_{i}}+\Vert (E-I_{k} E)'\Vert_{I_{i}}^{2}\le C\varepsilon^{-1}(N^{-1}\ln N)^{2k+1}.
\end{equation*}
On the other hand, one derives
\begin{equation*}
\Vert \eta''\Vert_{I_{i}}^{2}\le \Vert (S-I_{k} S)''\Vert^{2}_{I_{i}}+\Vert (E-I_{k} E)''\Vert_{I_{i}}^{2}\le C\varepsilon^{-3}(N^{-1}\ln N)^{2k-1},
\end{equation*}
where \eqref{eq:decomposition} and \eqref{eq:interpolation-theory} have been applied. Combined with the above derivation, we have 
\begin{equation}\label{eq: FF-3}
\{\eta'(x_{i})\}^{2}\le C\varepsilon^{-2}(N^{-1}\ln N)^{2k}.
\end{equation}
In particular, when $i=N/2$, one has
\begin{equation*}
\begin{aligned}
\{\eta'(x_{N/2})\}^{2}&\le h_{N/2}^{-1}\Vert \eta' \Vert_{I_{N/2}}^{2}+\Vert \eta' \Vert_{I_{N/2}}\Vert \eta'' \Vert_{I_{N/2}}\\
&+h_{N/2+1}^{-1}\Vert \eta' \Vert_{I_{N/2+1}}^{2}+\Vert \eta' \Vert_{I_{N/2+1}}\Vert \eta'' \Vert_{I_{N/2+1}}.
\end{aligned}
\end{equation*}
Here applying \eqref{eq: FF-1}, \eqref{eq: FF-2} and \eqref{eq: FF-3}, it is obvious to derive 
\begin{equation*}
\{\eta(x_{N/2})\}^{2}\le C\varepsilon^{-2}(N^{-1}\ln N)^{2k}.
\end{equation*}
So far, we have proved this conclusion.
\end{proof}

Next introduce $\chi:=\Pi u-u_{N}$ and recall $\eta: =\Pi u-u$.
From \eqref{eq:coercity} and the Galerkin orthogonality, we have
\begin{equation}\label{eq:uniform-convergence-1}
\begin{split}
&\Vert \chi \Vert_{NIPG}^2 \le B(\chi,\chi)=B(\Pi u-u+u-u_{N},\chi) =B(\eta,\chi)\\
& =\sum_{i=1}^{N}\int_{I_{i}}\varepsilon \eta'\chi'\mr{d}x-\varepsilon \sum_{i=0}^{N}\{\eta'(x_{i})\}[\chi(x_{i})]+\varepsilon \sum_{i=0}^{N}\{\chi'(x_{i})\}[\eta(x_{i})]\\
&+\sum_{i=0}^{N}\rho(x_{i})[\eta(x_{i})][\chi(x_{i})]+\sum_{i=1}^{N}\int_{I_{i}}a(x)\eta'\chi\mr{d}x-\sum_{i=0}^{N-1}a(x_{i})[\eta(x_{i})]\chi(x_{i}^{+})\\
&+ \sum_{i=1}^{N}\int_{I_{i}}b(x)\eta\chi\mr{d}x\\
&=\sum_{i=1}^{N}\int_{I_{i}}\varepsilon \eta'\chi'\mr{d}x-\varepsilon \sum_{i=0}^{N}\{\eta'(x_{i})\}[\chi(x_{i})]+\varepsilon \sum_{i=0}^{N}\{\chi'(x_{i})\}[\eta(x_{i})]\\
&+\sum_{i=0}^{N}\rho(x_{i})[\eta(x_{i})][\chi(x_{i})]-\sum_{i=1}^{N}\int_{I_{i}}a(x) \eta\chi'\mr{d}x+\sum_{i=1}^{N}a(x_{i})[\chi(x_{i})]\eta(x_{i}^{-})\\
&-\sum_{i=1}^{N}\int_{I_{i}}a'(x) \eta\chi\mr{d}x + \sum_{i=1}^{N}\int_{I_{i}}b(x)\eta\chi\mr{d}x\\
& =:\mr{I}+\mr{II}+\mr{III}+\mr{IV}+\mr{V}+\mr{VI}+\mr{VII}+\mr{VIII}.
\end{split}
\end{equation}
Now we will analyze the terms on the right-hand side of \eqref{eq:uniform-convergence-1}. First, we  may decompose $\mr{I}$ as
\begin{equation*}
\mr{I}=\sum_{i=1}^{N/2}\int_{I_{i}}\varepsilon \eta'\chi'\mr{d}x+\sum_{i=N/2+1}^{N}\int_{I_{i}}\varepsilon \eta'\chi'\mr{d}x.
\end{equation*}
For $i=1, \cdots, N/2$, from H\"{o}lder inequalities and \eqref{eq:QQ-3}, we obtain 
\begin{equation}\label{eq:convergence-1-1}
\begin{aligned}
\vert\sum_{i=1}^{N/2}\int_{I_{i}}\varepsilon \eta'\chi'\mr{d}x\vert&\le C\left(\sum_{i=1}^{N/2}\varepsilon\Vert\eta'\Vert^{2}_{I_{i}}\right)^{\frac{1}{2}}\left(\sum_{i=1}^{N/2}\varepsilon\Vert\chi'\Vert^{2}_{I_{i}}\right)^{\frac{1}{2}}\\
&\le C\left(\varepsilon^{\frac{1}{2}}N^{-k}+N^{-\sigma}+\varepsilon^{\frac{1}{2}}N^{1-\sigma}\right)\Vert\chi\Vert_{NIPG}.
\end{aligned}
\end{equation}
Furthermore, for $i=N/2+1, \cdots, N$, using H\"{o}lder inequalities and \eqref{eq:interpolation-theory}, a direct calculation shows that
\begin{equation}\label{eq:convergence-1-2}
\vert\sum_{i=N/2+1}^{N}\int_{I_{i}}\varepsilon (S-I_{k}S)'\chi'\mr{d}x\vert
\le C\varepsilon^{k+1}N^{-k}(\ln N)^{k+\frac{1}{2}}\Vert\chi\Vert_{NIPG}.
\end{equation}
Specially, through \eqref{eq: Gauss-Lobatto-1}, the following estimate can be derived.
\begin{equation}\label{eq:convergence-1-3}
\begin{aligned}
&\vert\sum_{i=N/2+1}^{N}\int_{I_{i}}\varepsilon(E-I_{k} E)'\chi'\mr{d}x\vert\\
&\le C\sum_{i=N/2+1}^{N}\varepsilon h_{i}^{k+1}\Vert E^{(k+2)}\Vert_{I_{i}}\Vert\chi'\Vert_{I_{i}}\\
&\le C\left(\sum_{i=N/2+1}^{N}\varepsilon (\varepsilon N^{-1}\ln N)^{2(k+1)}\Vert E^{(k+2)}\Vert^{2}_{I_{i}}\right)^{\frac{1}{2}}\left(\sum_{i=N/2+1}^{N}\varepsilon\Vert\chi'\Vert^{2}_{I_{i}}\right)^{\frac{1}{2}}\\
&\le \left(C\varepsilon^{-1}(N^{-1}\ln N)^{2(k+1)}\int_{1-\tau}^{1}e^{-2\beta(1-x)/\varepsilon}\mr{d}x\right)^{\frac{1}{2}}\Vert\chi\Vert_{NIPG}\\
&\le CN^{-(k+1)}(\ln N)^{k+1}\Vert\chi\Vert_{NIPG}.
\end{aligned}
\end{equation}
Combining \eqref{eq:convergence-1-1}, \eqref{eq:convergence-1-2} and \eqref{eq:convergence-1-3}, one derives
\begin{equation}\label{eq:convergence-1}
\mr{I}\le \left(C\varepsilon^{\frac{1}{2}}N^{-k}+CN^{-\sigma}+C\varepsilon^{\frac{1}{2}}NN^{-\sigma}
+CN^{-(k+1)}(\ln N)^{k+1}\right)\Vert\chi\Vert_{NIPG}.
\end{equation}

For $\mr{II}$, according to \eqref{eq:KK-1} and recalling $\sigma \ge k+1$ and $\varepsilon\le CN^{-1}$, we have
\begin{equation}\label{convergence-2}
\begin{aligned}
\vert\mr{II}\vert
&=\vert\varepsilon \sum_{i=0}^{N}\{\eta'(x_{i})\}[\chi(x_{i})]\vert\\
&\le \left(\sum_{i=0}^{N}\frac{\varepsilon^{2}}{\rho(x_{i})}\{\eta'(x_{i})\}^{2}\right)^{\frac{1}{2}}\left(\sum_{i=0}^{N}\rho(x_{i})[\chi(x_{i})]^{2}\right)^{\frac{1}{2}}\\
&\le \left(\sum_{i=0}^{N/2-1}\frac{\varepsilon^{2}}{\rho(x_{i})}\{\eta'(x_{i})\}^{2}+\sum_{i=N/2}^{N}\frac{\varepsilon^{2}}{\rho(x_{i})}\{\eta'(x_{i})\}^{2}\right)^{\frac{1}{2}}\Vert \chi\Vert_{NIPG}\\
&\le \left(CN^{-(2k+1)}+CN^{-1}(N^{-1}\ln N)^{2k}\right)^{\frac{1}{2}}\Vert \chi\Vert_{NIPG}\\
&\le CN^{-(k+\frac{1}{2})}(\ln N)^{k}\Vert \chi\Vert_{NIPG},
\end{aligned}
\end{equation}
where $\rho(x_{i})$ is defined as \eqref{eq: penalization parameters}.

Next divide $\mr{III}$ into two parts 
\begin{equation*}
\mr{III}=-\varepsilon \sum_{i=0}^{N/2-1}\{\chi'(x_{i})\}[\eta(x_{i})]-\varepsilon \sum_{i=N/2}^{N}\{\chi'(x_{i})\}[\eta(x_{i})].
\end{equation*}
According to \eqref{eq:H-1}, Remark \ref{special} and the properties of Gau{\ss} Lobatto interpolation, we just analyze $-\varepsilon \sum_{i=0}^{N/2-1}\{\chi'(x_{i})\}[\eta(x_{i})]$.
\begin{equation}\label{convergence-3}
\begin{aligned}
&\vert-\varepsilon \sum_{i=0}^{N/2-1}\{\chi'(x_{i})\}[\eta(x_{i})]\vert\\
&\le \vert\varepsilon \{\chi'(x_{0})\}[\eta(x_{0})]\vert +\vert\varepsilon \sum_{i=1}^{N/2-1}\{\chi'(x_{i})\}[\eta(x_{i})]\vert\\
&\le C\varepsilon\Vert \eta\Vert_{L^{\infty}(I_{1})}\Vert \chi'\Vert_{L^{\infty}(I_{1})}+C\varepsilon \sum_{i=1}^{N/2-1}\Vert \eta\Vert_{L^{\infty}(I_{i}\cup I_{i+1})}\Vert \chi'\Vert_{L^{\infty}(I_{i}\cup I_{i+1})}\\
&\le C\varepsilon \Vert \eta\Vert_{L^{\infty}(I_{1})}N^{\frac{1}{2}}\Vert \chi'\Vert_{I_{1}}+C\varepsilon \Vert \eta\Vert_{L^{\infty}(I_{i}\cup I_{i+1})}N^{\frac{1}{2}}\sum_{i=1}^{N/2-1}\Vert\chi'\Vert_{I_{i}\cup I_{i+1}}\\
&\le C\varepsilon^{\frac{1}{2}}N^{\frac{1}{2}}N^{-(k+1)}\Vert \chi\Vert_{NIPG}+C\varepsilon^{\frac{1}{2}}N^{-(k+\frac{1}{2})}N^{\frac{1}{2}}\Vert\chi\Vert_{NIPG}\\
&\le C\left(\varepsilon^{\frac{1}{2}}N^{-(k+\frac{1}{2})}+\varepsilon^{\frac{1}{2}} N^{-k}\right)\Vert\chi\Vert_{NIPG}\\
&\le C\varepsilon^{\frac{1}{2}}N^{-k}\Vert\chi\Vert_{NIPG},
\end{aligned}
\end{equation}
where the inverse inequality and \eqref{eq:QQ-5} have been employed.

For $\mr{IV}$, first, let's divide it into the following two parts:
\begin{equation*}
\sum_{i=0}^{N}\rho(x_{i})[\eta(x_{i})][\chi(x_{i})]=\sum_{i=0}^{N/2-1}\rho(x_{i})[\eta(x_{i})][\chi(x_{i})]+\sum_{i=N/2}^{N}\rho(x_{i})[\eta(x_{i})][\chi(x_{i})].
\end{equation*} 
Here we just analyze the first term. That is because using the nature of Gau{\ss} Lobatto interpolation on $[x_{N/2},1]$, combined with Remark \ref{special}, we have $[\eta(x_{i})]=0, i=N/2, N/2+1, \cdots, N$. Then from \eqref{eq: penalization parameters} and \eqref{eq:QQ-5}, 
\begin{equation}\label{convergence-4}
\begin{aligned}
&\vert\sum_{i=0}^{N/2-1}\rho(x_{i})[\eta(x_{i})][\chi(x_{i})]\vert\\
&\le \left(\sum_{i=0}^{N/2-1}\rho(x_{i})[\eta(x_{i})]^{2}\right)^{\frac{1}{2}}\left(\sum_{i=0}^{N/2-1}\rho(x_{i})[\chi(x_{i})]^{2}\right)^{\frac{1}{2}}\\
&\le C \left(\rho(x_{0})\Vert\eta\Vert_{L^{\infty}(I_{1})}^{2}+\sum_{i=1}^{N/2-1}\rho(x_{i})\Vert \eta\Vert_{L^{\infty}(I_{i}\cup I_{i+1})}^{2}\right)^{\frac{1}{2}}\Vert\chi\Vert_{NIPG}\\
&\le C\left(N^{-2(k+1)}+ N N^{-2(k+1)}\right)^{\frac{1}{2}}\Vert \chi\Vert_{NIPG}\\
&\le CN^{-(k+\frac{1}{2})}\Vert\chi\Vert_{NIPG}.
\end{aligned}
\end{equation}

Now we consider $\mr{V}$ and $\mr{VI}$, which are also divided into two parts $1\le i\le N/2$ and $N/2\le i\le N$ for analysis. For $1\le i\le N/2$, applying Remark \ref{special}, \eqref{eq:J-1} and \eqref{eq:J-2}, we have
\begin{equation*}
\begin{aligned}
&-\sum_{i=1}^{N/2}\int_{I_{i}}a(x) \eta\chi'\mr{d}x-\sum_{i=1}^{N/2}a(x_{i})[\chi(x_{i})]\eta(x_{i}^{-})\\
&=-\sum_{i=1}^{N/2}\int_{I_{i}}\left(a(x)-a(x_{i-\frac{1}{2}})\right)\eta\chi'\mr{d}x-\sum_{i=1}^{N/2}\int_{I_{i}}a(x_{i-\frac{1}{2}})\eta\chi'\mr{d}x-\sum_{i=1}^{N/2}a(x_{i})[\chi(x_{i})]\eta(x_{i}^{-})\\
&=-\sum_{i=1}^{N/2}\int_{I_{i}}\left(a(x)-a(x_{i-\frac{1}{2}})\right)\eta\chi'\mr{d}x
\end{aligned}
\end{equation*}
where $a(x_{i-\frac{1}{2}})$ is the value of $a(x)$ at the midpoint $x_{i-\frac{1}{2}}$ in the interval $[x_{i-1}, x_{i}]$. Then from the Lagrange mean value theorem, there exists $\xi$ between $x$ and $x_{i-\frac{1}{2}}$ such that
\begin{equation*}
a(x)-a(x_{i-\frac{1}{2}})=a'(\xi)(x-x_{i-\frac{1}{2}}).
\end{equation*}
Note that $a(x)$ in this paper is a smooth function. Therefore, we have
\begin{equation*}
\begin{aligned}
&\vert-\sum_{i=1}^{N/2}\int_{I_{i}}\left(a(x)-a(x_{i-\frac{1}{2}})\right)\eta\chi'\mr{d}x\vert=\vert-\sum_{i=1}^{N/2}\int_{I_{i}}a'(\xi)(x-x_{i-\frac{1}{2}})\eta\chi'\mr{d}x\vert\\
&\le C\sum_{i=1}^{N/2}h_{i}\Vert \eta\Vert_{L^{\infty}(I_{i})}\Vert\chi'\Vert_{L^{1}(I_{i})}\le C\sum_{i=1}^{N/2}h_{i}\Vert \eta\Vert_{L^{\infty}(I_{i})}N^{\frac{1}{2}}\Vert\chi\Vert_{I_{i}}\\
&\le C\sum_{i=1}^{N/2}N^{-\frac{1}{2}}\Vert \eta\Vert_{L^{\infty}(I_{i})}\Vert\chi\Vert_{I_{i}}\\
&\le CN^{-(k+\frac{3}{2})}\left(\sum_{i=1}^{N/2}1^{2}\right)^{\frac{1}{2}}\left(\sum_{i=1}^{N/2}\Vert\chi\Vert^{2}_{I_{i}}\right)^{\frac{1}{2}}\\
&\le CN^{-(k+1)}\Vert \chi\Vert_{NIPG},
\end{aligned}
\end{equation*}
where the inverse inequality and \eqref{eq:QQ-5} have been used.

For $N/2+1\le i\le N$, remembering $[\eta(x_{i})]=0$ and $\eta(x_{i})=0$, therefore, we need to analyze the following two estimates
\begin{equation*}
-\sum_{i=N/2+1}^{N}\int_{I_{i}}a(x) \eta\chi'\mr{d}x-\sum_{i=N/2+1}^{N}a(x_{i})[\chi(x_{i})]\eta(x_{i}^{-}).
\end{equation*}
Adopting the inverse inequality, \eqref{eq:QQ-6} and H\"{o}lder inequalities, there is
\begin{equation*}
\begin{aligned}
&\vert-\sum_{i=N/2+1}^{N}\int_{I_{i}}a(x)\eta\chi'\mr{d}x\vert\le C\sum_{i=N/2+1}^{N}\Vert \eta\Vert_{L^{\infty}(I_{i})}\Vert\chi'\Vert_{L^{1}(I_{i})}\\
&\le C\Vert \eta\Vert_{L^{\infty}(I_{i})}\sum_{i=N/2+1}^{N}\varepsilon^{\frac{1}{2}}(N^{-1}\ln N)^{\frac{1}{2}}\Vert\chi'\Vert_{I_{i}}\\
&\le C(N^{-1}\ln N)^{k+\frac{3}{2}} \left(\sum_{i=N/2+1}^{N} 1^{2}\right)^{\frac{1}{2}}\left(\sum_{i=N/2+1}^{N}\varepsilon\Vert\chi'\Vert_{I_{i}}^{2}\right)^{\frac{1}{2}}\\
&\le CN^{-(k+1)}(\ln N)^{k+\frac{3}{2}}\Vert\chi\Vert_{NIPG}.
\end{aligned}
\end{equation*}
Besides, from \eqref{eq:QQ-6} and \eqref{eq: penalization parameters} we can obtain
\begin{equation*}
\begin{aligned}
&\vert-\sum_{i=N/2+1}^{N}a(x_{i})[\chi(x_{i})]\eta(x_{i}^{-})\vert\\&\le C\left(\sum_{i=N/2+1}^{N}\rho^{-1}(x_{i})\eta(x_{i}^{-})^{2}\right)^{\frac{1}{2}}\left(\sum_{i=N/2+1}^{N}\rho(x_{i})[\chi(x_{i})]^{2}\right)^{\frac{1}{2}}\\
&\le C\left(\sum_{i=N/2+1}^{N}\rho^{-1}(x_{i})\Vert\eta\Vert^{2}_{L^{\infty}(I_{i})}\right)^{\frac{1}{2}}\Vert\chi\Vert_{NIPG}\\
&\le CN^{-(k+\frac{3}{2})}(\ln N)^{k+1}\Vert\chi\Vert_{NIPG}.
\end{aligned}
\end{equation*}
Summing up, we can derive the following estimate directly.
\begin{equation}\label{convergence-5}
\mr{V}+\mr{VI}\le CN^{-(k+1)}(\ln N)^{k+\frac{3}{2}}\Vert\chi\Vert_{NIPG}.
\end{equation}

For $\mr{VII}$ and $\mr{VIII}$, by means of H\"{o}lder inequalities, \eqref{eq:QQ-1} and \eqref{eq:QQ-2}, one obtains
\begin{equation}\label{convergence-6}
\begin{aligned}
\mr{VII}+\mr{VIII}\le C\Vert\eta\Vert_{[0, 1]}\Vert\chi\Vert_{NIPG}\le C\left(\varepsilon^{\frac{1}{2}}(N^{-1}\ln N)^{k+1}+N^{-(k+1)}\right)\Vert\chi\Vert_{NIPG}
\end{aligned}
\end{equation}

Last but not least, by \eqref{eq:convergence-1}, \eqref{convergence-2}, \eqref{convergence-3}, \eqref{convergence-4}, \eqref{convergence-5} and \eqref{convergence-6}, there is
\begin{equation*}
\begin{aligned}
\Vert\chi\Vert^{2}_{NIPG}&\le\mr{I}+\mr{II}+\mr{III}+\mr{IV}+\mr{V}+\mr{VI}+\mr{VII}+\mr{VIII}\\
&\le \left(CN^{-(k+\frac{1}{2})}(\ln N)^{k} + CN^{-(k+1)}(\ln N)^{k+\frac{3}{2}}\right)\Vert\chi\Vert_{NIPG},
\end{aligned}
\end{equation*}
which implies 
\begin{equation}\label{eq:MA}
\Vert \Pi u-u_{N}\Vert_{NIPG}\le CN^{-(k+\frac{1}{2})}(\ln N)^{k}+CN^{-(k+1)}(\ln N)^{k+\frac{3}{2}}
\end{equation}

Now we show the main conclusion of supercloseness.

\begin{theorem}\label{the:main result1}
Let Assumption \ref{ass:S-1} hold. Let the mesh $\{x_i\}$ be Shishkin mesh defined in \eqref{eq:Shishkin mesh-Roos} with $\sigma\ge k+1$. $\rho(x_{i})$ is defined as \eqref{eq: penalization parameters} and the definition of Gau{\ss} Lobatto interpolation $I_{k}u$ is presented in Section 3.1.
Then we derive
\begin{align*}
\Vert I_{k}u-u_{N}\Vert_{NIPG}+\Vert \Pi u-u_{N} \Vert_{NIPG}\le CN^{-(k+\frac{1}{2})}(\ln N)^{k}+CN^{-(k+1)}(\ln N)^{k+\frac{3}{2}},
\end{align*}
where $\Pi u$ is the interpolation of the exact solution of \eqref{eq:S-1}, while $u_N$ is the solution of \eqref{eq:SD}. 
\end{theorem}
\begin{proof}
First, from the triangle inequality, there is
\begin{equation*}
\begin{aligned}
\Vert I_{k}u-u_{N}\Vert_{NIPG}\le \Vert I_{k}u-\Pi u\Vert_{NIPG}+\Vert \Pi u-u_{N}\Vert_{NIPG}.
\end{aligned}
\end{equation*}
From \eqref{eq:MA}, the estimate of $\Vert \Pi u-u_{N} \Vert_{NIPG}$ has been known, therefore, we just analyze the bound of $\Vert I_{k}u-\Pi u\Vert_{NIPG}$ in the following.

According to the definition of interpolation $\Pi u$ and the triangle inequality, we have
\begin{equation*}
\begin{aligned}
\Vert I_{k}u-\Pi u\Vert_{NIPG}&=\Vert I_{k}u-I_{k}u\Vert_{NIPG, [x_{N/2}, 1]}+\Vert I_{k}u-P_{h}^{-}u\Vert_{NIPG, [0, x_{N/2}]}\\
&=\Vert I_{k}u-P_{h}^{-}u\Vert_{NIPG, [0, x_{N/2}]}\\
&\le \Vert I_{k}u-u\Vert_{NIPG, [0, x_{N/2}]}+\Vert u-P_{h}^{-}u\Vert_{NIPG, [0, x_{N/2}]}.
\end{aligned}
\end{equation*}
Then applying \eqref{post-process-1} and \eqref{post-process-2}, one derives
\begin{equation*}
\Vert I_{k}u-P_{h}^{-}u\Vert_{NIPG, [0, x_{N/2}]}\le CN^{-(k+\frac{1}{2})}.
\end{equation*}
Thus the proof of this theorem has been completed.
\end{proof}

\begin{remark}
Below we explain why Gau{\ss} Radau interpolation is selected outside the layer. 
The main reason is that the convection term outside the layer
$$-\sum_{i=1}^{N/2}\int_{I_{i}}a(x) \mu\chi'\mr{d}x$$
cannot reach $k+1$ order by standard arguments, where $\mu=u-u_{I}$. Note that $u$ is the exact solution of the problem, while $u_{I}\in C(\bar{\Omega})$ is the usual Lagrange interpolation polynomial of order $k$ on $I_{i}, i=1, 2, \cdots, N$. Specifically, for $\Vert \chi'\Vert_{L^{1}(I_{i})}$, if the inverse estimate is used, we have
\begin{equation*}
\begin{aligned}
\vert-\sum_{i=1}^{N/2}\int_{I_{i}}a(x)\mu\chi'\mr{d}x\vert&\le C\Vert\mu\Vert_{L^{\infty}([0,1-\tau])}\Vert\chi'\Vert_{L^{1}([0,1-\tau])}\le CN^{-(k+1)}N\Vert\chi\Vert_{[0,1-\tau]}\\
&\le CN^{-k}\Vert\chi\Vert_{NIPG}.
\end{aligned}
\end{equation*}
Moreover, if directly converted to $\Vert \chi' \Vert_{I_{i}}$, there will be no the factor $\varepsilon^{\frac{1}{2}}$ outside the layer, which matches with $\Vert \chi'\Vert_{I_{i}}$ to become $\Vert\chi\Vert_{NIPG}$.
\begin{equation*}
\begin{aligned}
\vert-\sum_{i=1}^{N/2}\int_{I_{i}}a(x)\mu\chi'\mr{d}x\vert&\le C\Vert\mu\Vert_{L^{\infty}([0,1-\tau])}\Vert\chi'\Vert_{L^{1}([0,1-\tau])}\le C\varepsilon^{-\frac{1}{2}}N^{-(k+1)}\varepsilon^{\frac{1}{2}}\Vert\chi'\Vert_{[0,1-\tau]}\\
&\le C\varepsilon^{-\frac{1}{2}}N^{-(k+1)}\Vert\chi\Vert_{NIPG}.
\end{aligned}
\end{equation*}
Therefore, here we introduce a special projection---Gau{\ss} Radau interpolation to solve this difficulty. 
\end{remark}
 
\begin{theorem}\label{eq:main result2}
Let the mesh $\{x_i\}$ be Shishkin mesh defined in \eqref{eq:Shishkin mesh-Roos} with $\sigma\ge k+\frac{5}{2}$. Assume that $\rho(x_{i})$ is defined as
 \begin{equation*}
\rho(x_{i})=\left\{
\begin{aligned}
&N^{-1},\quad 0\le i\le N/2-1,\\
&N^{3},\quad N/2\le i\le N
\end{aligned}
\right.
\end{equation*}
and $\varepsilon\le CN^{-2}$. Then we derive
\begin{align*}
\Vert I_{k}u-u_{N}\Vert_{NIPG}+\Vert \Pi u-u_{N} \Vert_{NIPG}\le CN^{-(k+1)}(\ln N)^{k+\frac{3}{2}},
\end{align*}
where $\Pi u$ is the interpolation of the exact solution of \eqref{eq:S-1}, while $u_N$ is the solution of \eqref{eq:SD}. 
\end{theorem}
\begin{proof}
According to the proof process of Theorem \ref{the:main result1}, it is obvious that when the above conditions are met, the estimates of $\mr{II}$, $\mr{IV}$ and $\mr{V}+\mr{VI}$ can be improved to order $k+1$. Furthermore, under the above assumptions, it is straightforward to obtain that
\begin{equation*}
\Vert I_{k}u-\Pi u\Vert_{NIPG}\le CN^{-(k+1)}.
\end{equation*}
Hence, we can draw this conclusion directly.
\end{proof} 
\section{Post-processing technology}
Combined with the supercloseness results obtained in Section 3.2, we show how to improve the accuracy of the numerical solution by constructing a post-processing operator $R$. More specifically, $\Vert u-R u_{N}\Vert_{NIPG}\ll \Vert u-u_{N}\Vert_{NIPG}$.

Suppose $N$ is a positive integer that can be divided by $4$. Then we can construct a coarser mesh $\mathcal{T}_{N/2}$ composed of disjoint macro elements $M$, in which each macro cell $M$ is composed of two adjacent cells  on the original mesh $\mathcal{T}_{N}$. Note that each $M$ belongs to only one of the two subdomains $[0, 1-\tau] $ and $[1-\tau, 1]$ of $\Omega$. For the mesh points $\Omega_{N/2} = \{x_{2j}, j=1, 2, \cdots, N/2\}$ on $\mathcal{T}_{N/2}$,  we define $M_ {j}=[x_{2(j-1)}, x_{2j}]$, whose length is $H_ {j}=x_{2j}-x_{2(j-1)}$. 

According to the $k+1$ Gau{\ss} Lobatto points $\{t_{s}\}_{s=0}^{k}$ defined on $I\in \mathcal{T}_{N}$ in Section 3.1, it is straightforward to derive that there are $2k+1$ Gau{\ss} Lobatto points $\{t_{jm}\}_{m=0}^{2k}$ on each cell $M_{j}, j=1, 2, \cdots, N/2$ on $\mathcal{T}_{N/2}$.

Next, select $k+2$ points in each macro cell $M_{j}, j=1, 2, \cdots, N/2$ to define an interpolation operator $R_{M_j} : C(M_{j}\setminus \{x_{2j-1}\}) \rightarrow \mathcal{P}_{k+1}(M_{j})$, which satisfies
\begin{equation}\label{P-1}
\begin{aligned}
&R_{M_j}v(t_{jm})= v(t_{jm}), \quad \text{for $m\in G$ , if $k$ is even},\\
&
\left\{
\begin{aligned}
&R_{M_j}v(t_{jm})= v(t_{jm}),\; m\in G\setminus\{ k\},\\
&R_{M_j}v(t_{jk})=\frac{v(t_{jk}^{+})+v(t_{jk}^{-})}{2}
\end{aligned}
\right.\quad \text{$$ , if $k$ is odd}.
\end{aligned}
\end{equation}
with $G=\{0, 1, 3, 5, \cdots, 2k-1, 2k\}$. 
Generally, $R_{M_j}$ can be extended to a global operator by setting 
$$(Rv)(x)= R_{M_j}v(x),\quad \forall x\in M_j,\; j=1, 2, \cdots, N/2.$$
\begin{lemma}\label{post processing}
Assume that  $\rho(x_{j})\ge C\max\{\varepsilon N, N^{-1}\}, j=0, 1, \cdots, N/2-1$, while $\rho(x_{j})\ge CN(\ln N)^{-1}, j=N/2, \cdots, N$. Then  for the post-processing operator $R$ and Gau{\ss} Lobatto interpolation operator $I_{k}$, we have the following properties
\begin{align}
&Rv= R I_{k}v, \quad \forall  v\in C(\bar{\Omega}),\label{eq:P-2}\\
&\Vert Rv\Vert_{NIPG}\le C\Vert v\Vert_{NIPG},\quad \forall v\in V^{N}\label{eq:P-3}.
\end{align}

\end{lemma}
\begin{proof}
Applying the definition of $R$ and considering $v\in C(\bar{\Omega})$, there is
\begin{equation*}
R I_{k}v(t_{jm})= I_{k}v(t_{jm})=v(t_{jm}), \quad \forall m\in G.
\end{equation*}
\eqref{eq:P-2} can be obtained easily.

From the definition of NIPG norm, one has
\begin{equation*}
\begin{aligned}
\Vert Rv\Vert_{NIPG}^{2}&=\varepsilon \sum_{j=1}^{N/2}\Vert (Rv)'\Vert_{M_{j}}^{2}+\sum_{j=1}^{N/2}\gamma\Vert Rv\Vert_{M_{j}}^{2}+\sum_{j=0}^{N/2}\rho(x_{2j})[R v(x_{2j})]^{2}\\
&=\textcircled{1}+\textcircled{2}+\textcircled{3}.
\end{aligned}
\end{equation*}
Here for \textcircled{3}, the following formula can be derived directly by the definition of $R$.
\begin{equation*}
\sum_{j=0}^{N/2}\rho(x_{2j})[Rv(x_{2j})]^{2}=\sum_{j=0}^{N/2}\rho(x_{2j})[v(x_{2j})]^{2}.
\end{equation*}

For \textcircled{1}, we first analyze 
$\varepsilon\Vert (Rv)'\Vert_{M_{j}}^{2}$. Before that, we present the following equality for all $v\in V^{N}$.
\begin{equation}\label{eq:P-4}
v(t_{jm})=\int_{x}^{t_{jm}}v'(s)\mr{d}s+A[v(x_{2j-1})]+v(x),
\quad x\in M_j.
\end{equation}
Since $v(x)$ may be discontinuous at $x_{2j-1}$, $A$ may take $0, 1, -1$ for any fixed $x\in M_j$. On the basis of that,  H\"{o}lder inequality and \eqref{eq:P-4} yield
\begin{equation*}
\begin{aligned}
&\varepsilon\Vert (Rv)'\Vert_{M_{j}}^{2}=\varepsilon\int_{M_{j}}[(Rv)']^{2}\mr{d}x=\varepsilon\int_{M_{j}}(\sum_{m\in G}Rv(t_{jm})\phi'_{m,k+1}(x))^{2}\mr{d}x\\
&=\varepsilon\int_{M_{j}}(\sum_{m\in G}v(t_{jm})\phi'_{m,k+1}(x))^{2}\mr{d}x\\
&=\varepsilon\int_{M_{j}}\left(\sum_{m\in G}\left(\int_{x}^{t_{jm}}v'(s)\mr{d}s+A[v(x_{2j-1})]\right)\phi'_{m,k+1}(x)+ v(x)\sum_{m\in G}\phi'_{m,k+1}(x)\right)^{2}\mr{d}x\\
&=\varepsilon\int_{M_{j}}\left(\sum_{m\in G}\int_{x}^{t_{jm}}v'(s)\mr{d}s \phi'_{m,k+1}(x)+\sum_{m\in G}A[v(x_{2j-1})]\phi'_{m,k+1}(x)\right)^{2}\mr{d}x\\
&\le 2\varepsilon\int_{M_{j}}\left\{\left(\sum_{m\in G}\int_{x}^{t_{jm}}v'(s)\mr{d}s\phi'_{m, k+1}(x)\right)^{2}+\left(\sum_{m\in G}A[v(x_{2j-1})]\phi'_{m, k+1}(x)\right)^{2}\right\}\mr{d}x,
\end{aligned}
\end{equation*}
where $\sum\limits_{m\in G}\phi_{m,k+1}(x)\equiv 1$ for $x\in M_j$ has been known. Note that $x_{2j-1}$ represents the midpoint of each macro cell $M_{j}, j=1, 2, \cdots, N/2$ and $\phi_{m, k+1}(x)$ is the interpolation basis function at Gau{\ss} Lobatto points $\{t_{jm}\}_{m\in G}$ of interval $M_{j}$. Then by means of Cauchy Schwartz inequality and the value range of $\rho(x_{2j-1})$, we have 
\begin{equation*}
\begin{aligned}
&\varepsilon\Vert (Rv)'\Vert_{M_{j}}^{2}\le 2\varepsilon\int_{M_{j}}\left\{\sum_{m\in G}\left(\int_{x}^{t_{jm}}v'(s)\mr{d}s\right)^{2}+\sum_{m\in G}[v(x_{2j-1})]^{2}\right\}\sum_{m\in G}\left(\phi'_{m, k+1}(x)\right)^{2}\mr{d}x\\
&\le 2\varepsilon\int_{M_{j}}\left\{\sum_{m\in G}\left(\int_{M_{j}}v'(s)\mr{d}s\right)^{2}+\sum_{m\in G}[v(x_{2j-1})]^{2}\right\}\sum_{m\in G}\left(\phi'_{m, k+1}(x)\right)^{2}\mr{d}x\\
&\le 2\varepsilon\left\{\sum_{m\in G}H_{j}\Vert v'\Vert_{M_{j}}^{2}+\sum_{m\in G}[v(x_{2j-1})]^{2}\right\}\sum_{m\in G}\int_{M_{j}}\left(\phi'_{m, k+1}(x)\right)^{2}\mr{d}x \quad\text{(by using variable substitution)}\\
&\le C\varepsilon\left\{\sum_{m\in G}H_{j}\Vert v'\Vert_{M_{j}}^{2}+\sum_{m\in G}[v(x_{2j-1})]^{2}\right\}\sum_{m\in G}H_{j}^{-1}\int_{-1}^{1}\left(\hat{\phi}'_{m, k+1}(\hat{x})\right)^{2}\mr{d}\hat{x }\\
&\le C\varepsilon\Vert v'\Vert_{M_{j}}^{2}+C\varepsilon H_{j}^{-1}\sum_{m\in G}[v(x_{2j-1})]^{2}\\
&\le C\varepsilon\Vert v'\Vert_{M_{j}}^{2}+C\rho(x_{2j-1})[v(x_{2j-1})]^{2}.
\end{aligned}
\end{equation*}

Applying the same method as above, for \textcircled{2}, we have
\begin{equation*}
\Vert Rv\Vert_{M_{j}}^{2}=\int_{M_{j}}(Rv)^{2}\mr{d}x\le C\Vert v\Vert_{M_{j}}^{2}+C\rho(x_{2j-1})[v(x_{2j-1})]^{2},
\end{equation*}
where   H\"{o}lder inequality, \eqref{eq:P-4} and the inverse inequality had used.

In summary, we can easily arrive at the conclusion of this lemma.
\end{proof}
\begin{lemma}\label{post-1}
Let Assumption \ref{ass:S-1} hold true and $\sigma \ge k+1$. Then  there is
\begin{equation*}
\Vert u-Ru\Vert_{NIPG}\le C\varepsilon^{\frac{1}{2}}N^{1-\sigma}+CN^{-(k+1)}(\ln N)^{k+1},
\end{equation*}
where $R$ is the post-processing operator defined at the beginning of this Section, and $\rho(x_{j})$ meets the condition in Lemma \ref{post processing}.
\end{lemma}
\begin{proof}
In the light of the definition of the NIPG norm \eqref{eq:SS-1}, one derives
\begin{equation*}
\begin{aligned}
\Vert u-Ru\Vert_{NIPG}^{2}&=\varepsilon \sum_{j=1}^{N/2}\Vert (u-Ru)'\Vert_{M_{j}}^{2}+\sum_{j=1}^{N/2}\gamma \Vert u-Ru\Vert_{M_{j}}^{2}+\sum_{j=0}^{N/2}\rho(x_{2j})[(u-Ru)(x_{2j})]^{2}\\
&=\Theta_{1}+\Theta_{2}+\Theta_{3}.
\end{aligned}
\end{equation*}
For $\Theta_{3}$, from \eqref{P-1}, it is clearly $\Theta_{3}=0$.

For $\Theta_{1}$, employ the decomposition of the solution, we divide it into two parts. More specifically,
\begin{equation*}
\Theta_{1}\le 2\varepsilon \sum_{j=1}^{N/2}\Vert (S-RS)'\Vert_{M_{j}}^{2}+2\varepsilon \sum_{j=1}^{N/2}\Vert (E-RE)'\Vert_{M_{j}}^{2},
\end{equation*}
where there is
\begin{equation*}
\vert\varepsilon \sum_{j=1}^{N/2}\Vert (S-RS)'\Vert_{M_{j}}^{2}\vert\le C\varepsilon N^{-2(k+1)}.
\end{equation*}
This is because $Ru\in \mathcal{P}_{k+1}$ is a polynomial of $k+1$, so it can be derived directly in a general way. In addition, from the general interpolation theory \eqref{eq:interpolation-theory} and the triangle inequality, one has
\begin{equation*}
\begin{aligned}
&\vert\varepsilon \sum_{j=1}^{N/2}\Vert (E-RE)'\Vert_{M_{j}}^{2}\vert\\
 &\le C\varepsilon \sum_{j=1}^{N/4}\left(\Vert(E)'\Vert_{M_{j}}^{2}+\Vert (RE)'\Vert_{M_{j}}^{2}\right)+C\varepsilon \sum_{j=N/4+1}^{N/2}H_{j}^{2(k+1)}\Vert E^{(k+1)}\Vert _{M_{j}}^{2}\\
&\le C\varepsilon^{-1}\int_{0}^{x_{N/2}}e^{-2\beta(1-x)/\varepsilon}\mr{d}x+C\varepsilon \sum_{j=1}^{N/4}H_{j}^{-1}\Vert RE\Vert_{L^{\infty}(M_{j})}^{2}+C(N^{-1}\ln N)^{2(k+1)}\\
&\le CN^{-2\sigma}+C\varepsilon N^{2} N^{-2\sigma}+C(N^{-1}\ln N)^{2(k+1)}\\
&\le C\varepsilon N^{2} N^{-2\sigma}+C(N^{-1}\ln N)^{2(k+1)}.
\end{aligned}
\end{equation*}

For $\Theta_{2}$, through the same method as above, we have
\begin{equation*}
\Theta_{2}\le CN^{-2(k+2)}+CN^{-2\sigma}+C\varepsilon (N^{-1}\ln N)^{2(k+2)}.
\end{equation*}
So far, we have completed the derivation of this lemma.
\end{proof}
\begin{theorem}
Assume that Assumption \ref{ass:S-1} holds true, $\sigma\ge k+1$ and $\rho(x_{i}), i=0, 1, \cdots, N$ are defined in \eqref{eq: penalization parameters}, then
\begin{equation*}
\Vert u-R u_{N}\Vert_{NIPG}\le CN^{-(k+\frac{1}{2})}(\ln N)^{k}+CN^{-(k+1)}(\ln N)^{k+\frac{3}{2}},
\end{equation*}
where $u$ is the exact solution of the problem \eqref{eq:S-1}, and $R u_{N}$ represents the post processed numerical solution.
\end{theorem}
\begin{proof}
According to the triangle inequality, \eqref{eq:P-2} and \eqref{eq:P-3}, we have
\begin{equation*}
\begin{aligned}
\Vert u-R u_{N}\Vert_{NIPG}&\le \Vert u-Ru \Vert_{NIPG}+\Vert Ru-R I_{k}u\Vert_{NIPG}+\Vert R I_{k}u-R u_{N}\Vert_{NIPG}\\
&\le \Vert u-Ru \Vert_{NIPG}+C\Vert I_{k}u-u_{N}\Vert_{NIPG}.
\end{aligned}
\end{equation*}
Through Lemma \ref{post-1} and Theorem \ref{the:main result1}, the proof of this theorem has been completed.
\end{proof}
\begin{remark}
Specially, if $\varepsilon\le N^{-2}, \sigma\ge k+\frac{5}{2}$ and
\begin{equation*}
\rho(x_{i})=\left\{
\begin{aligned}
&N^{-1},\quad 0\le i\le N/2-1,\\
&N^{3},\quad N/2\le i\le N,
\end{aligned}
\right.
\end{equation*}
the following estimate hold true,
\begin{equation*}
\Vert u-Ru\Vert_{NIPG}\le CN^{-(k+1)}(\ln N)^{k+1}.
\end{equation*}
Furthermore, from Theorem \ref{eq:main result2}, there is
\begin{equation*}
\Vert I_{k}u-u_{N}\Vert_{NIPG}+\Vert \Pi u-u_{N}\Vert_{NIPG}\le CN^{-(k+1)}(\ln N)^{k+\frac{3}{2}}.
\end{equation*}
Therefore, it is straightforward to derive
\begin{equation*}
\Vert u-Ru_{N}\Vert_{NIPG}\le CN^{-(k+1)}(\ln N)^{k+\frac{3}{2}}.
\end{equation*}
\end{remark}

\section{Superconvergence}
In particular, in order to obtain the uniform superconvergence, according to the bilinear form $B(\cdot,\cdot)$, the discrete NIPG norm \cite{Zha1:2002-F} is denoted as 
\begin{equation*}
\Vert v \Vert_{\varepsilon, NIPG}^{2}:=
\varepsilon \sum_{i=1}^{N}h_{i}\sum_{j=1}^{k}w_{j}v'(x_{ij})^{2}+\sum_{i=1}^{N}\gamma\Vert v \Vert_{I_{i}}^{2}+\sum_{i=0}^{N}\rho(x_{i})[v(x_{i})]^{2} \quad \forall v\in V_{N}^{k},
\end{equation*}
where $\{x_{ij}\}_{j=1}^{k}$ are the set of the Gaussian points in $I_{i} = (x_{i-1}, x_{i})$, $w_{j} > 0$ are weights for the $k$-point Gaussian quadrature rule and $\rho(x_{i})$ are the penalty parameters defined in \eqref{eq: penalization parameters}. 
\begin{remark}\label{note}
Since the $k$-point Gaussian quadrature rule is exact for any algebraic polynomial no more than $2k - 1$ degree, therefore, for all $v\in V_{N}^{k}$ we have 
\begin{equation*}
\Vert v\Vert_{NIPG}=\Vert v\Vert_{\varepsilon, NIPG}.
\end{equation*}
Furthermore, by employing \eqref{eq:SPP-condition-1}, the following coercivity holds
\begin{equation*}
B(v_N,v_N) \ge \Vert v_N \Vert_{\varepsilon, NIPG}^2\quad \text{for all $v_N\in V_{N}^{k}$}.
\end{equation*}
It implies that \eqref{eq:SD} has a unique solution $u_{N}$.
 On the basis of that, according to Theorem \ref{the:main result1}, there is
\begin{equation*}
\Vert I_{k}u-u_{N}\Vert_{\varepsilon, NIPG}+\Vert \Pi u-u_{N}\Vert_{\varepsilon, NIPG}\le CN^{-(k+\frac{1}{2})}(\ln N)^{k}+ CN^{-(k+1)}(\ln N)^{k+\frac{3}{2}}.
\end{equation*}
In a similar way, when the conditions of Theorem \ref{eq:main result2} are satisfied, one has
\begin{equation*}
\Vert I_{k}u-u_{N}\Vert_{\varepsilon, NIPG}+\Vert \Pi u-u_{N}\Vert_{\varepsilon, NIPG}\le CN^{-(k+1)}(\ln N)^{k+\frac{3}{2}}.
\end{equation*}
\end{remark}

In order to obtain the relevant superconvergence result, we present the following lemma.
\begin{lemma}\label{calculation}
If $\omega\in H^{k+2}(I)$ and $I_{k} \omega\in \mathcal{P}_{k}$ are defined in Section 3.1, then outside the layer we have 
$$\vert\varepsilon \sum_{i=N/2+1}^{N}h_{i}\sum_{j=1}^{k}w_{j}(\omega-I_{k} \omega)'(x_{ij})^{2}\vert\le C \sum_{i=N/2+1}^{N}\varepsilon h_{i} h_{i}^{2k+2} \Vert \omega^{(k+2)}\Vert^{2}_{L^{\infty}(I_{i})}.$$
\end{lemma}
\begin{proof}
Applying the arguments in \citep[Lemma 3.3]{Zha1:2002-F}, we obtain this conclusion obviously.
\end{proof}

From Lemma \ref{calculation} and \eqref{eq:interpolation-theory}, the following interpolation error estimates can be derived directly. 
\begin{lemma}\label{eq:interpolation-error}
Let Assumption \ref{ass:S-1} hold and $\sigma\ge k+1$. For Shishkin mesh \eqref{eq:Shishkin mesh-Roos}, there is  
\begin{align}
&\Vert u-I_{k} u\Vert_{\varepsilon, NIPG,[1-\tau,1]}\le C(N^{-1}\ln N)^{k+1},\label{eq:interpolation-error-1}\\
&\Vert u-I_{k}u\Vert_{\varepsilon, NIPG,[0,1-\tau]}\le C\varepsilon^{\frac{1}{2}}N^{-k}+CN^{-(k+1)} \label{eq:interpolation-error-2}.
\end{align}
\end{lemma}
\begin{proof}
Applying the similar arguments in \cite{Zha1:2002-F}, we can obtain this lemma easily.
\end{proof}
\begin{corollary}\label{cor:Sss}
Let Assumption \ref{ass:S-1} and $\sigma \ge k+1$ hold. For Shishkin mesh \eqref{eq:Shishkin mesh-Roos}, one has  
\begin{align*}
\Vert u-I_{k} u\Vert_{\varepsilon, NIPG} \le C\varepsilon^{\frac{1}{2}}N^{-k}+CN^{-(k+1)}(\ln N)^{k+1}.
\end{align*}
If $\varepsilon\le CN^{-2}$, there is
\begin{equation*}
\Vert u-I_{k} u\Vert_{\varepsilon, NIPG}\le CN^{-(k+1)}(\ln N)^{k+1}.
\end{equation*}
\end{corollary}
\begin{proof}
Collecting \eqref{eq:interpolation-error-1} and \eqref{eq:interpolation-error-2}, this corollary can be obtained.
\end{proof}
 
Now we provide the main conclusion of superconvergence.

\begin{theorem}\label{the:main result3}
Suppose that Assumption \ref{ass:S-1} hold. Let the mesh $\{x_i\}$ be Shishkin mesh defined in \eqref{eq:Shishkin mesh-Roos} with $\sigma\ge k+1$ and $\rho(x_{i}), i=1, 2, \cdots, N$ are defined in \eqref{eq: penalization parameters}. 
Then we derive
\begin{align*}
\Vert u-u_{N} \Vert_{\varepsilon, NIPG}\le CN^{-(k+\frac{1}{2})}(\ln N)^{k}+ CN^{-(k+1)}(\ln N)^{k+\frac{3}{2}},
\end{align*}
Moreover, when $\varepsilon\le CN^{-2}$, $\sigma\ge k+\frac{5}{2}$ and $\rho(x_{i})$ is defined as
 \begin{equation*}
\rho(x_{i})=\left\{
\begin{aligned}
&N^{-1},\quad 0\le i\le N/2-1,\\
&N^{3},\quad N/2\le i\le N,
\end{aligned}
\right.
\end{equation*} 
the following eatimate holds true.
\begin{align*}
\Vert u-u_{N} \Vert_{\varepsilon, NIPG}\le CN^{-(k+1)}(\ln N)^{k+\frac{3}{2}},
\end{align*}
where $u$ is the exact solution of \eqref{eq:S-1}, while $u_N$ is the solution of \eqref{eq:SD}. 
\end{theorem}
\begin{proof}
From a triangle inequality, Remark \ref{note} and Corollary \ref{cor:Sss}, one has
\begin{equation*}
\begin{aligned}
\Vert u-u_{N}\Vert_{\varepsilon, NIPG}&\le \Vert u-I_{k} u\Vert_{\varepsilon, NIPG}+\Vert I_{k} u-u_{N}\Vert_{\varepsilon, NIPG}.
\end{aligned}
\end{equation*}
Thus it is straightforward to obtain the conclusion of this theorem.
\end{proof}
\section{Numerical experiment}
In this  section, we verify the previous theoretical conclusions about supercloseness by considering a singularly perturbed convection diffusion problem.
\begin{equation}\label{eq:KK-2}
\left\{
\begin{aligned}
 &-\varepsilon u''(x)+(3-x)u'(x)+u(x)=f(x),\quad x\in \Omega: = (0,1),\\
&u(0)=u(1)=0.
\end{aligned}
\right.
\end{equation}
where $f(x)$ is chosen such that
\begin{equation*}
u(x)=x(1-e^{2(1-x)/\varepsilon}),
\end{equation*}
is the exact solution of the \eqref{eq:KK-2}.

For our numerical experiment we consider $\varepsilon= 10^{-8}, 10^{-9},10^{-10} ,10^{-11}, k = 1, 2, 3, 4, 5$ and $N =8, 16, 32, 64$. 
Besides, for Shishkin mesh \eqref{eq:Shishkin mesh-Roos} we take $\beta=2$, $\sigma = k + \frac{5}{2}$ and
\begin{equation*}
\rho(x_{i})=\left\{
\begin{aligned}
&N^{-1},\quad 0\le i\le N/2-1,\\
&N^{3},\quad N/2\le i\le N.
\end{aligned}
\right.
\end{equation*}
Then, the corresponding convergence rate is defined by
$$p^{N}_{I}= \frac{\ln e^{N}_{I}-\ln e^{2N}_{I}}{\ln 2},$$
where
$e^{N}_{I}= \Vert I_{k}u-u_{N}\Vert_{NIPG}$ is the computation error with $N$ number of interval for a particular $\varepsilon$. 
Below, we provide the numerical results in the following figures, which  show the correctness of Theorem \ref{eq:main result2}.\\
\includegraphics[width=1.2\textwidth]{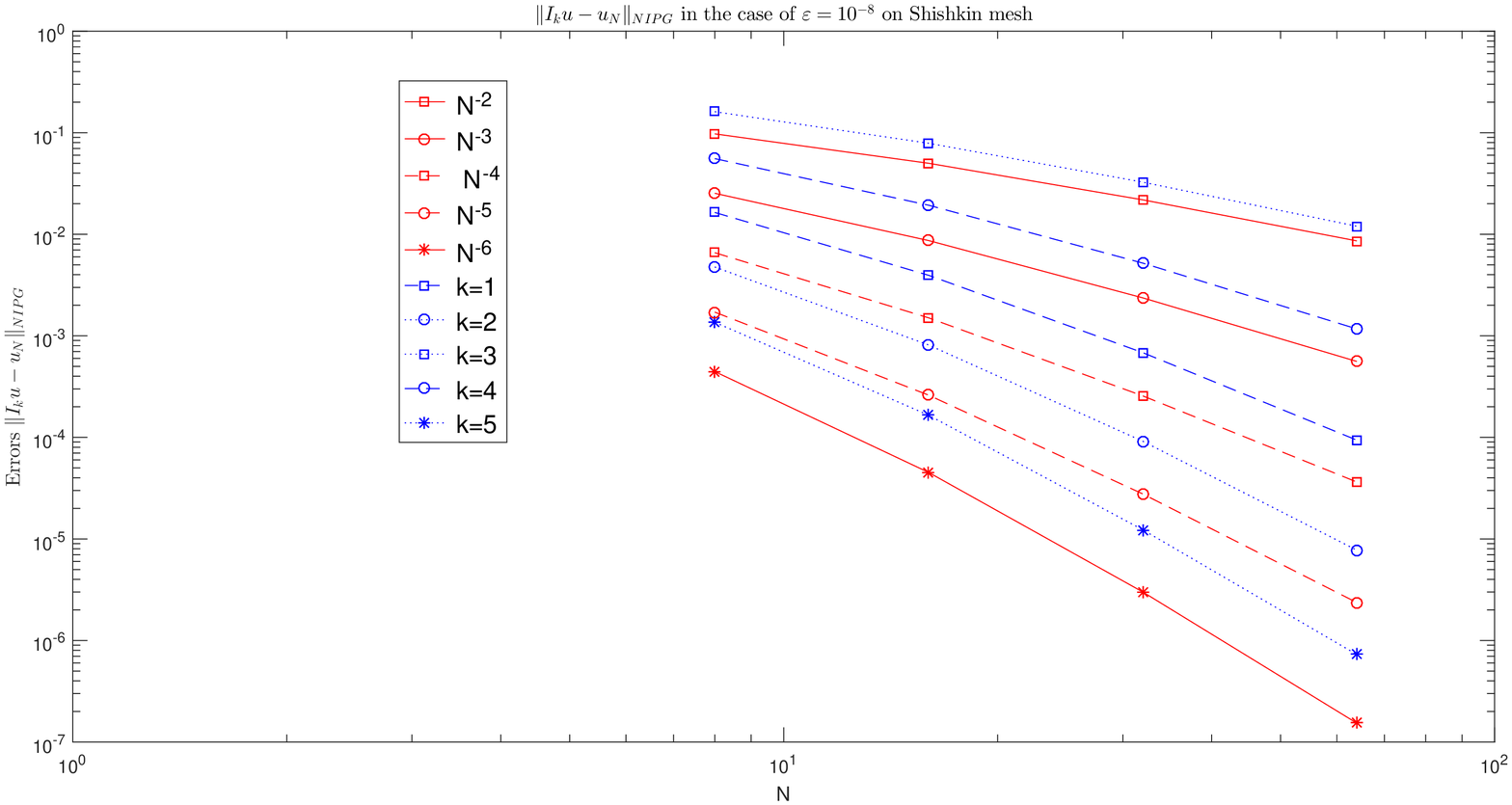}\\
\includegraphics[width=1.2\textwidth]{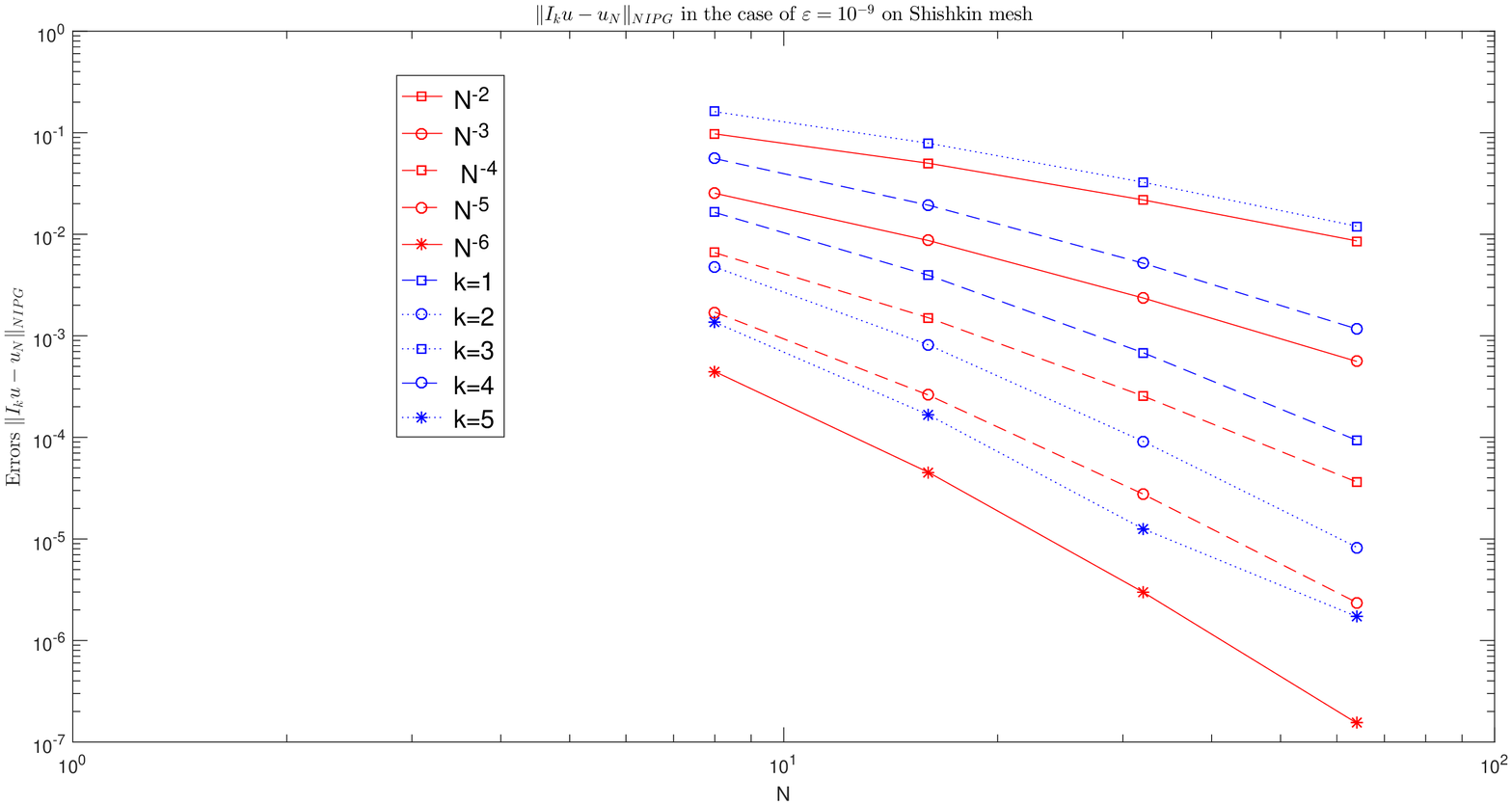}\\
\includegraphics[width=1.2\textwidth]{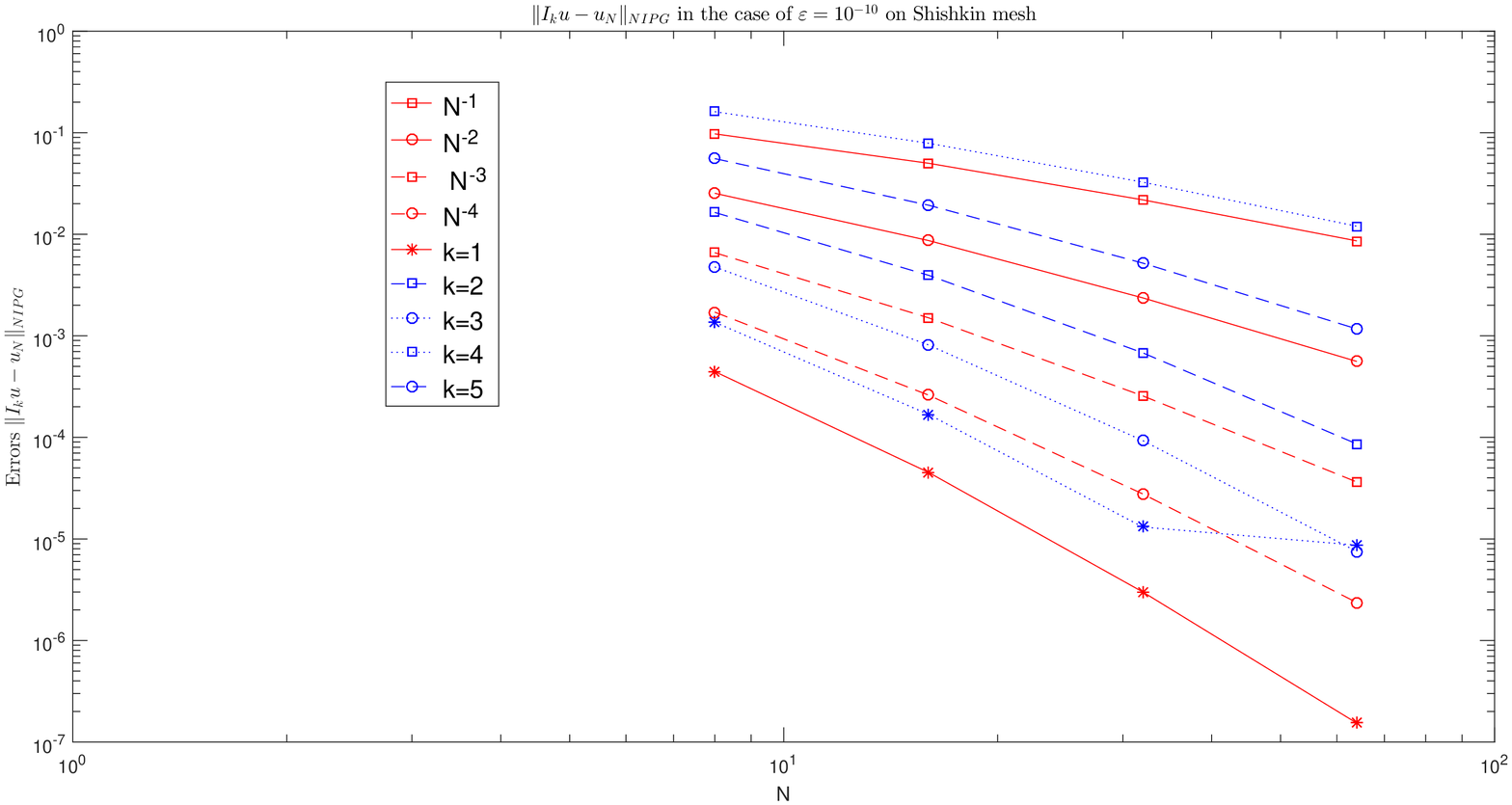}\\
\includegraphics[width=1.2\textwidth]{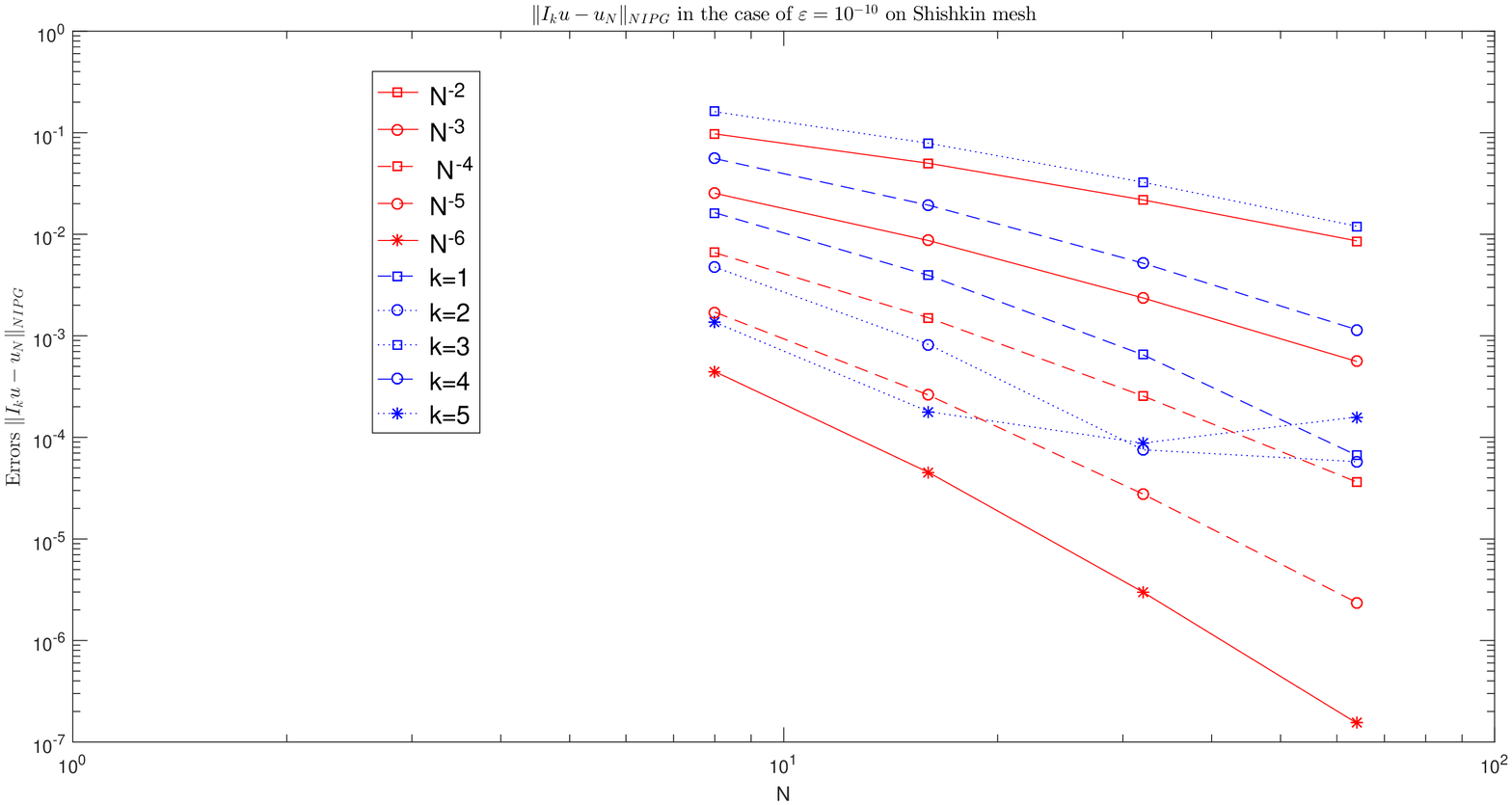}

Furthermore, we observe that the numerical results are unstable when the the degree of polynomial $k$ or the number of interval $N$ is large. One possible reason is that the linear system becomes ill conditioned as $\varepsilon\rightarrow 0$ or $k$ and $N$ become large. The details are shown in the figures above.
%

\section{Declarations}
 \subsection{Funding}
This research is supported by   
National Natural Science Foundation of China (11771257,11601251).
 \subsection{Data availability statement}
The authors confirm that the data supporting the findings of this study are available within the article and its supplementary materials.

\subsection{Conflict of interests}
The authors declare that they have no conflict of interest.

\end{document}